\documentclass[10pt,draftcls,onecolumn]{IEEEtran}
\pdfoutput=1 

\usepackage{cite}
\usepackage{amsmath,amssymb,amsfonts,dsfont}
\usepackage{mathrsfs}
\usepackage{algorithm}
\usepackage{algorithmic}
\usepackage{graphicx}
\usepackage{textcomp}
\usepackage{comment}
\usepackage{bm}
\usepackage{mathtools}

\usepackage{amsthm}   % can cause trouble in arXiv
\usepackage{xcolor}

\usepackage{diagbox}

\usepackage{array}

\usepackage{url}

\usepackage{subcaption}

\usepackage{float}
\usepackage{epsfig}
\usepackage{mathrsfs}
\usepackage{longtable}
\usepackage{xtab}
\usepackage{ifthen}
\usepackage{thmtools,thm-restate}
\usepackage[thinlines]{easytable}
\usepackage{pifont}% http://ctan.org/pkg/pifont
\usepackage{enumerate}
\usepackage{calrsfs}  % make mathcal fancier
\newboolean{showcomments}
\setboolean{showcomments}{true}
\newcommand{\lina}[1]{  \ifthenelse{\boolean{showcomments}}
	{ \textcolor{red}{(Lina says:  #1)}} {}  }

\newcommand{\nbf}{\noindent\textbf}
\newcommand{\nit}{\noindent\textit}

\DeclarePairedDelimiter\floor{\lfloor}{\rfloor}
\DeclarePairedDelimiter\ceil{\lceil}{\rceil}

\DeclareMathOperator*{\argmin}{arg\,min}
\DeclareMathOperator*{\E}{\mathbb{E}}

\newcommand{\R}{\mathbb{R}}
\newcommand{\F}{\mathcal F}

\newcommand{\Pb}{\mathbb P}

\newcommand{\A}{\mathcal{A}}
\newtheorem{example}{Example}
\newtheorem{remark}{Remark}

\newtheorem{assumption}{Assumption}
\allowdisplaybreaks

\newtheorem{lemma}{Lemma}
\newtheorem{theorem}{Theorem}
\newtheorem{corollary}{Corollary}
\newtheorem{proposition}{Proposition}

\allowdisplaybreaks

\def\BibTeX{{\rm B\kern-.05em{\sc i\kern-.025em b}\kern-.08em
		T\kern-.1667em\lower.7ex\hbox{E}\kern-.125emX}}

\allowdisplaybreaks

\begin{document}
\title{Online Optimization with Predictions and Switching Costs: Fast Algorithms and the Fundamental Limit}
\author{Yingying Li, Guannan Qu, and Na Li % <-this % stops a space
\thanks{The work was supported by NSF 1608509, NSF CAREER 1553407, AFOSR YIP, and ARPA-E through the NODES program.
	 Y. Li, G. Qu and N. Li are with John A. Paulson School of Engineering and Applied Sciences, Harvard University, 33 Oxford Street, Cambridge, MA 02138, USA (email: yingyingli@g.harvard.edu, gqu@g.harvard.edu, nali@seas.harvard.edu). A preliminary version of this paper appears in the 2018 American Control Conference. The current version contains a new algorithm with better performance and refined fundamental limit results.}
%\thanks{Manuscript received April 19, 2005; revised December 27, 2012.}
}
%\author{Michael~Shell,~\IEEEmembership{Member,~IEEE,}
   %     John~Doe,~\IEEEmembership{Fellow,~OSA,}
 %       and~Jane~Doe,~\IEEEmembership{Life~Fellow,~IEEE}% <-this % stops a space
%\thanks{M. Shell was with the Department
%of Electrical and Computer Engineering, Georgia Institute of Technology, Atlanta,


\maketitle

\begin{abstract}
	This paper studies an online optimization problem with a finite prediction window of cost functions and additional switching costs on decisions. {We propose two gradient-based online algorithms:  Receding Horizon Gradient Descent (RHGD), and Receding Horizon Accelerated Gradient (RHAG).} Both algorithms only require a finite number of projected gradient evaluations at each stage.  {We provide upper bounds on the dynamic regrets of the 
	proposed algorithms and show that the regret upper bounds decay exponentially  with the length of the prediction window. Moreover, we study the fundamental lower bound on the dynamic regret for a broad class of deterministic online algorithms. The lower bound is close to RHAG's regret upper bound, indicating that our gradient-based RHAG is a near-optimal online algorithm. Finally, we conduct numerical experiments to complement our theoretical analysis.}
\end{abstract}

\section{Introduction}
%\blue{First write and finalize the main text. In the end, write the introduction before the contribution. The introduction is the most difficult part and also one of  the most important part. }

% {This is a new problem formulation, so the introduction is very important. Also, I should carefully explain how this paper relates with the control community to attract their attention from the control community. I should clarify the possible doubts.}

In many applications of sequential decision-making problems, e.g.  data center scheduling \cite{lin2012online}, smart grids \cite{tanaka2006real}, multi-task machine learning \cite{zenke2017continual}, autonomous driving \cite{rios2016survey}, etc., the system is subject to a time-varying environment and only the near future can be predicted with high accuracy; meanwhile, abrupt and large changes in the decisions  are undesirable. Inspired by this, in this paper, we  consider an online convex optimization problem with a short-term prediction of the cost functions and additional switching costs on the decisions. In particular, at each stage $t$, an agent receives the cost functions for the next $W$ stages, i.e. $f_t, \dots, f_{t+W-1}$, and then makes a decision $x_t$. The agent suffers the stage cost $f_t(x_t)$ and also a switching cost $\frac{\beta}{2}\|x_t-x_{t-1}\|^2$ that penalizes the change in the decision at stage $t$, where $\beta\geq 0$ is a weight parameter.

%the stage cost functions $f_t(\cdot), \dots, f_{t+W-1}(\cdot)$ at the next $W$ stages are  available to a decision maker who will make a decision $x_t$ and then suffer the stage cost $f_t(x_t)$ plus a switching cost $\frac{\beta}{2}\|x_t-x_{t-1}\|^2$ where $\beta\geq 0$ is a weighting parameter.

%Motivated by these  applications, we  consider a variant of online convex optimization \cite{hazan2016introduction} with predictions and switching costs, In particular,  at each stage $t$,  the stage cost functions $f_t(\cdot), \dots, f_{t+W-1}(\cdot)$ in the next $W$ stages are  accurately predicted and revealed to an decision maker, based on which the decision maker takes an action $x_t$ and suffers the current stage cost $f_t(x_t)$ plus a switching cost $\frac{\beta}{2}\|x_t-x_{t-1}\|^2$ that penalizes the changes of the action from the previous stage with a penalty parameter $\beta\geq 0$.

%The problem described above has attracted a lot of attention recently in online optimization, and many algorithms have been proposed in literature, e.g. AFHC CHC. 

This problem is a variant of the classical online convex optimization  (OCO) that does not consider predictions or switching costs \cite{hazan2016introduction}. This variant  has attracted a lot of attention from the OCO community in recent years. Many algorithms have been proposed and their optimality/regret guarantees have been analyzed, e.g. AFHC \cite{lin2012online}, CHC \cite{chen2016using}.   {In addition, though not specially designed for this problem, a classical control algorithm, model predictive control (MPC) \cite{rawlings2012postface},  can  be naturally applied here because MPC considers the receding horizon optimization for the next $W$ stages and the information of the next $W$ stages is available in our setting.} However, most  algorithms in literature, e.g. AFHC, CHC, MPC, require to solve $W$-stage optimization at each stage, which can be time-consuming for large-scale problems.  {Moreover, despite the attempts on relieving the computational burden of MPC, e.g. explicit MPC \cite{alessio2009survey}, inexact MPC \cite{kogel2014stabilization}, suboptimal MPC \cite{wang2010fast,graichen2010stability}, prediction-correction methods \cite{paternain2018prediction}, real-time iteration schemes \cite{diehl2005nominal}, etc., most existing methods either require special structures  of the cost functions or lack optimality/regret guarantees for the online problems with time-varying costs.}

%   the online problem described above shares some similarities with the Economic Model Predictive Control (EMPC) problem with time-varying costs,  which aims to minimize the average of the time-varying economic  costs  by considering the receding horizon optimization  for the next $W$ stages.  EMPC algorithms can  be naturally applied to our online problem. 

%   both the OCO and the control community. 
  
%   in recent years, many OCO algorithms have been proposed and analyzed for this problem, e.g. (cite??). Moreover, notice that the online problem  described above is similar to the Economic Model Predictive Control problems with time-varying costs, and thus  EMPC algorithms can be applied here as well. 
  
%   have witnessed a growing interest on this problem from the OCO community and many online optimization algorithms have been proposed, e.g. AFHC, CHC, etc. (cite??). 
  
%   %the problem has attracted a lot of attention  in recent years and several online optimization algorithms have been proposed, e.g. AFHC, CHC, etc. (cite??). 

% In addition, the online problem described above is closely related with the Economic Model Predictive Control problems with time-varying costs, and thus  EMPC algorithms can be applied here as well. 

Further,  questions regarding the fundamental limit of this problem are relatively under-explored in the literature.  {In this paper, we are mostly interested in the following question. \\
\textit{Q: what is achievable and what is not by online algorithms that have a finite  window of the cost predictions? }}

\nbf{Contributions:}
This paper  considers strongly convex and smooth cost functions and  studies the performance by \textit{dynamic regret}, which is the online  algorithm's cost minus the optimal cost in hindsight \cite{mokhtari2016online}. Our results are summarized below.

% {need revision, repeated meaning}

Firstly, we design gradient-based online algorithms, RHGD and RHAG. Our algorithms  adopt classical OCO algorithms (without predictions), e.g. online gradient descent (OGD) \cite{hazan2016introduction}, as an initialization oracle, and then performs gradient updates on the initialization based on the cost predictions. With OGD initialization, our algorithms only require the calculation of $W+1$ gradients. Moreover, we show theoretically that our algorithms reduce the regret of the initialization oracle, which can be any classical OCO algorithm, exponentially with the prediction window's length $W$.  We also provide some stability guarantees  of our algorithms under certain conditions. 

 In addition, we analyze the fundamental lower bounds  on the dynamic regrets of any online algorithms, which is based on the worst-case analysis and imposes  no computational constraints on  the admissible algorithms. Surprisingly, the regret upper bound of RHAG is close to the fundamental lower bound, indicating that, at least in the worst cases, our gradient-based RHAG algorithm is near-optimal even when compared with more computationally intensive algorithms. 
 
 %, at least in the worst cases, the regret performance of our gradient-based RHAG is comparable with  the online algorithms employing more computational power. 

 %in terms of the dependence on $W$ and  the path length under some conditions.
 % {}
 
%  {to revise}

Finally, we  conduct numerical experiments to complement our theoretical analysis by comparing our algorithms with MPC and discuss  the implications of the numerical results.

 \nbf{Related work:} We provide a brief literature review below.
 
%  {To revise: do not make it too strong, MPC also cares about optimality, mention MPC and EMPC with time-varying costs, and mention optimality and transient optimality. Then, we don't need to explain why our optimal control is not EMPC.}

\noindent {\nit{\underline{(Economic) model predictive control.}} The problem considered in this paper is  related with the (economic) model predictive control ((E)MPC) with time-varying costs \cite{ellis2014economic,ferramosca2014economic,angeli2016theoretical,alessandretti2016convergence,grune2018economic,zavala2010real,zanon2013lyapunov}, where EMPC is a variant of MPC that focuses on minimizing the economic costs (see e.g. \cite{angeli2016theoretical}).  In addition, our dynamic regret analysis is  related with the optimality performance analysis of (E)MPC that  studies how  MPC's cost  deviates from the  optimal one \cite{ferramosca2014economic,angeli2016theoretical,alessandretti2016convergence,grune2018economic}. However, the optimality performance of the fast (E)MPC  schemes (see e.g. \cite{alessio2009survey,kogel2014stabilization,wang2010fast,graichen2010stability,paternain2018prediction,diehl2005nominal}) is relatively under-explored, especially for the time-varying cases.} %since most literature focuses on the numerical and stability performance of 

 {It is worth mentioning that MPC can be applied to much more general problems, e.g. dynamical control problems,  systems with constraints, distributed systems, etc. (see \cite{rawlings2012postface} for more details), while this paper only considers OCO with switching costs. Nevertheless, this paper may lay  foundation   for more general problems, e.g. the online optimal control with a linear time-invariant system as in \cite{li2019online}. }

 \nit{\underline{Online convex optimization.}} For classical OCO, we refer the reader to \cite{hazan2016introduction}. The OCO with switching costs  has been studied in \cite{lin2012online,chen2016using,goel2019online}. The OCO with predictions has been studied for  the case without switching costs  where \textit{one-stage inaccurate predictions} are considered (see e.g.\cite{rakhlin2013online}), and for the case with switching costs and \textit{multi-stage predictions} which are either accurate \cite{lin2012online} or inaccurate  \cite{chen2016using}. This paper is mostly related with the  work on the $W$-stage accurate prediction \cite{lin2012online}. 
 %However,  there lacks fundamental limit analysis and gradient-based algorithm design in literature. 

% Common performance metrics include regret \cite{hazan2016introduction,mokhtari2016online} and competitive ratio \cite{andrew2013tale}.
 
 Various performance metrics are considered in  OCO, e.g. static regret \cite{hazan2016introduction}, dynamic regret \cite{mokhtari2016online,besbes2015non}, competitive ratio \cite{andrew2013tale}, etc. The static regret refers to the difference between the online algorithm's cost and the cost generated by an optimal \textit{static} decision in hindsight. There are algorithms with sublinear $o(T)$ static regret bounds \cite{hazan2016introduction}. However, with large fluctuations in the environment, it is also reasonable to compare with the  cost of the  possibly \textit{time-varying} optimal decisions in hindsight, as considered in the dynamic regret and the competitive ratio. This paper will focus on the dynamic regret and leave the competitive ratio analysis as the future work.  Notice that it is usually impossible to achieve a sublinear $o(T)$ dynamic regret bound for all time-varying environments and most dynamic regret bounds in literature depend on the variation of the environment   \cite{besbes2015non,mokhtari2016online}. 

\noindent  {\nit{\underline{Time-varying optimization (TVO).}} It considers  $\min_x f(x;t)$ for each $t$. For theoretical purposes, most papers on TVO assume that the  cost function $f(x;t)$ does not change dramatically with time $t$, e.g. $f(x;t)$ has certain smoothness  properties with respect to  $t$ \cite{simonetto2016class,tang2018running,zavala2010real}, which are not assumed in this paper. It is also worth mentioning the prediction-correction  method \cite{simonetto2016class}, which computes the predictions of the future costs based on the smoothness of  $f(x;t)$ with  $t$; while the predictions in this paper are not computed by our algorithms but are given by some outside sources. It is   our future work to consider designing algorithms to generate predictions.}

\nbf{Notations:}
  $\|\cdot\|$ denotes the $l_2$ norm. $\Pi_{\mathbb X}(x)$ denotes the projection of $x$ onto set $\mathbb X$.  For  function $f(x,y)$ on $(x,y)\in \mathbb R^{m+n}$, let $\nabla f(x,y)\in \mathbb R^{m+n}$ be the gradient and $\frac{\partial f}{\partial x}(x,y)\in \mathbb R^m$ be the partial gradient with  $x$.  {For integers $a, b, c$, we write $a\equiv b \pmod c$ if $a=b+kc$ for some integer $k$.} $|\mathbb J|$ denotes the cardinality of the set $\mathbb J$. $\mathbf A^\top $ denotes the matrix $\mathbf A$'s transpose.  {For $x>0$ in $\R$, 
we write  $f(x)=O(g(x))$ ($f(x)=\Omega(g(x))$)  if there exists a constant $M$ such that $|f(x) | \leq M g(x)$ ($|f(x) | \geq M g(x)$) for  $x \geq M$;  and we write $f(x)=o(g(x))$  if $\lim_{x\to +\infty} f(x)/g(x)=0$.} $\mathbf 1_n \in \R^n$ is an all-one vector. $\bm I_n$ is an identity matrix in $\R^{n \times n}$.

\section{Problem Formulation}
This paper considers an online convex optimization (OCO) problem in $T$ stages with stage cost function $f_t(\cdot)$ and  quadratic switching cost $\frac{\beta}{2}\|x_t-x_{t-1}\|^2$ to penalize the changes in the decisions. %\footnote{The analysis can be extended to other switching cost functions with properties  such as  convexity, and smoothness.} 
Formally, we aim to solve %the following online convex optimization problem:
\begin{equation}\label{equ: soco}
\begin{aligned}
\min_{x_1, \dots, x_T \in \mathbb X} \mathsf{C_T}(\bm x)=\sum_{t=1}^T \left(f_t(x_t)+ \frac{\beta}{2}\|x_t-x_{t-1}\|^2\right),
\end{aligned}
\end{equation}
where $\mathbb X\subseteq \R^n$ is a convex feasible set, $\bm x=(x_1, \dots, x_T)$, $x_0\in \mathbb X$ is given, $\beta\geq 0$ is a  penalty parameter. %Let $\bm x^*$ denote optimal solution to \eqref{equ: soco}.
%The SOCO problem  enjoys many applications \cite{lin2012online,lin2013dynamic}.  

To solve \eqref{equ: soco}, all   cost functions $f_1,\dots, f_T$ have to be known a priori, which is not practical in many applications \cite{lin2012online}. Nevertheless, there are usually some predictions available, especially for the near future. This paper adopts a simple  model to characterize the predictions: at each stage $t$, the decision maker receives the cost functions for the next $W$ stages $f_t, \dots, f_{t+W-1}$,\footnote{ {Predicting the complete  function can be challenging, but it simplifies the analysis and it is often practical when the cost functions are parametric \cite{lin2012online}.}} but do not know the cost functions  beyond the next $W$ stages, that is, $f_{t+W}, f_{t+W+1},\dots$ may be arbitrary or even adversarial.  Though our prediction model is too optimistic in the near future but too pessimistic in the far future, our model captures a commonly observed property  in  applications, i.e. the short-range predictions  are usually much more accurate than the long-range predictions. In addition, our model simplifies the theoretical analysis and helps generate insightful results that may lay foundation for future work on more realistic settings, e.g.  noisy and/or partial predictions. %In addition, this model has been considered in literature to serve as an initial step to more general settings \cite{chen2016using}. %, and some  results on this simplified model have been extended to more general settings (cite??). 

In summary, the online  problem considered in this paper is outlined below. At each stage $t=1,2, \dots, T$, an agent
\begin{itemize}
	\item  receives the predicted cost functions $f_t(\cdot), \dots, f_{t+W-1}(\cdot)$
	\item computes a stage decision $x_t$ by history and predictions% $f_1(\cdot), \dots, f_{t+W-1}(\cdot)$,
	\item suffers the cost $f_t(x_t)+ \frac{\beta}{2}\|x_t-x_{t-1}\|^2$.
\end{itemize}

The online information available at each stage $t$ contains the predicted cost functions as well as the history cost functions, i.e. $\{f_1, \dots,  f_{t+W-1}\}$.
Our goal is to design an online algorithm $\A$ that computes the stage decision $x_t^{\A}$ by only using the online information at  stage $t$ to minimize the total cost $\mathsf{C_T}(\bm x^{\A})$. We measure the algorithm performance by \textit{dynamic regret} \cite{mokhtari2016online}, which compares the online algorithm's cost with the optimal cost in hindsight:
\begin{equation}\label{equ: dynamic regret def}
    \textrm{Reg}(\A)=\mathsf{C_T}(\bm x^{\A})- \mathsf{C_T}(\bm x^*)
\end{equation}
where $\bm x^*$ denotes the optimal solution to \eqref{equ: soco} in hindsight.

To ease the theoretical analysis, we list a few assumptions on the cost functions $\{f_t\}$ and the feasible set $\mathbb X$.
\begin{assumption}\label{ass:1}   {Cost function $f_t(\cdot)$ is $\alpha_t$-strongly convex and $l_t$-smooth  in $\R^n$,\footnote{ {Here we consider $\R^n$ because we will use Nesterov's accelerated gradient that requires strong convexity and smoothness  outside the feasible set $\mathbb X$ \cite{nesterov2013introductory}.}} i.e. for any $x, y \in \R^n$, we have
$f_t(x)+\langle\nabla f_t(x) ,y-x\rangle +\frac{\alpha_t}{2}\|y-x\|^2\leq f_t(y)\leq f_t(x)+\langle\nabla f_t(x) ,y-x\rangle +\frac{l_t}{2}\|y-x\|^2 $.}

 {In addition, 
		there exist constants $\alpha,l>0$ that do not depend on $T$ such that  $\alpha_t \geq \alpha$ and $l_t \leq l$ for all $1\leq t \leq T$.}
		
% 		, i.e. for any $t$ and any $x, y \in \R^n$, we have 
% 		$$f_t(y)\geq f_t(x)+\langle\nabla f_t(x) ,y-x\rangle +\frac{\alpha}{2}\|y-x\|^2$$
% 		$$f_t(y)\leq f_t(x)+\langle\nabla f_t(x) ,y-x\rangle +\frac{l}{2}\|y-x\|^2 $$
		
% 		$f_t(\cdot)$ satisfies the following properties for any $1\leq t \leq T$.
% 		\begin{enumerate}
% 			\item[i)] Strong convexity:
% 			$$f_t(y)\geq f_t(x)+\langle\nabla f_t(x) ,y-x\rangle +\frac{\alpha}{2}\|y-x\|^2 , \ \forall\, x,y \in \R^n.\footnote{ {Here we consider $\R^n$ because Nesterov's accelerated gradient requires strong convexity and smoothness beyond outside the feasible set $\mathbb X$ \cite{nesterov2013introductory}.}}$$
% 			\item[ii)] Smoothness: 
% 			$$f_t(y)\leq f_t(x)+\langle\nabla f_t(x) ,y-x\rangle +\frac{l}{2}\|y-x\|^2 , \ \forall\, x,y \in \R^n.$$
% 			%\item[iii)] Bounded gradient on $\mathbb X$:
% 		%	$\| \nabla f_t(x)\|\leq G, \ \forall\, x\in \mathbb X.$
% 		\end{enumerate}
% 		%	\begin{align*}
		%	&\text{$\alpha$-strong convexity:} \\
		%	&\quad	f(y)\geq f(x)+\langle\nabla f(x) ,y-x\rangle +\frac{\alpha}{2}\|y-x\|^2 , \ \forall\, x,y \in \R^n\\
		%	&\text{$l$-smoothness:}\\
		%	& \quad  \|\nabla f(y)-\nabla f(x)\|\leq l\|y-x\|, \quad  \, \forall\, x,y \in \R^n\\
		%	& \text{Bounded gradient:} \quad \quad \| \nabla f(x)\|\leq G, \quad \forall\, x\in X
		%	\end{align*}
		%We denote the class of these functions as $ \mathcal F_X(\alpha,l,G)$. 
	\end{assumption}
	\begin{assumption}\label{ass: bdd gradient}
	         {There exists $G>0$ such that 	$\| \nabla f_t(x)\|\leq G$ for any $x\in \mathbb X$ and any $1\leq t\leq T$.}
	\end{assumption}
	\begin{assumption}\label{ass:2}
	  {$\mathbb X$ is compact with  $D\coloneqq\max_{x,y\in \mathbb X}\|x-y\|$.}
\end{assumption}
 {We assume Assumption \ref{ass:1} holds throughout the paper. We will explicitly state it when Assumption \ref{ass: bdd gradient} and \ref{ass:2} are needed.}

% {Notice that $\alpha$ in Assumption \ref{ass:1} is a uniform lower bound on the strong convexity factors of $f_1, \dots, f_T$. Similarly,  $l$ and $G$ are uniform upper bounds on the corresponding factors. }%($\mu$) can be viewed as the uniform upper (lower) bounds. 

% {It left as future work to relax the assumptions.}

Finally, we provide two examples for our problem above.

%where our problem formulation can be applied. 

% {Q: what's $\mathbb X$ in both cases?}

\begin{example}[Trajectory Tracking] Consider a dynamical system $x_{t+1}=x_{t}+u_{t}$, where $x_t$ is the robot's location,  $u_t$ is the robot's velocity. Let $y_t$ be the target's location. The optimal control problem with tracking error $f_t(x_t)=\frac{1}{2} \|x_t - y_t\|^2$ and  control cost $\frac{\beta}{2}\|u_t\|^2=\frac{\beta}{2}\|x_{t+1}-x_{t}\|^2$ can be formulated as % The objective is to minimize the sum of the tracking error and the energy loss,
	% so this is not identical to our model, because switching cost is from x_{t+1} to x_t, does this make a big difference??
	$$
	\min_{x_t \in \mathbb X} \sum_{t=0}^{T-1} \left(f_t(x_t)+\frac{\beta}{2}\|x_{t+1}-x_{t}\|^2\right)+f_T(x_T),$$
	where $\mathbb X$ denotes a feasible set. 
	In reality,  a short lookahead window  is sometimes available for the target trajectory \cite{rosales1998improved}. 
	\end{example}
	
	\begin{example}[ {Smoothed regression}]\label{example: regression} 
		 {Consider a learner who aims to solve  a sequence of  regression  tasks %where the for sequential tasks 
		without changing the regressors too much between stages (see e.g. \cite{goel2019online,zenke2017continual}). The problem can be modeled as \eqref{equ: soco} where $f_t(\cdot)$ represents the  regression loss function, $x_t$ is the regressor at stage $t$, and $\beta$ is the parameter for the smoothing regularization. In some cases,  a short lookahead window of future tasks are available, for example,  when multiple tasks arrive at the same time but are solved sequentially (see e.g. \cite{pentina2015curriculum}).}
	\end{example}

\section{Online Algorithm Design}
This section presents our online algorithms RHGD and RHAG, which are inspired by the offline optimization and offline gradient methods, i.e. gradient descent and Nesterov's accelerated gradient \cite{nesterov2013introductory} respectively.

%We first discuss the insights from the offline optimization and offline gradient descent, based on which we design our RHGD.  In addition,  we present RHAG which is based on  Nesterov's accelerated gradient \cite{nesterov2013introductory}. 
%In this section, we first introduce RHGD which is based on vanilla gradient descent, then RHAG is similar by applying the design idea to Nesterov's accelerated gradient \cite{nesterov2013introductory}. % Finally, we prove the dynamic regret upper bounds and the stability.

\subsection{Offline Optimization and Offline Gradient Descent}

%$\vec{\mathbb X}$

Given all cost functions $f_1,\ldots, f_T$, the problem \eqref{equ: soco} becomes a classical optimization problem and can be solved  by, e.g.,  projected gradient descent (GD). The updating rule of GD is the following.   For iteration $k=1, 2, \dots$,
\begin{equation}\label{equ: gd x}
\bm x{(k)}=\Pi_{\mathbb X\times \dots \times \mathbb X}\left(\bm x{(k-1)}-\eta \nabla \mathsf{C_T}(\bm x{(k-1)})\right),
\end{equation}
where $\mathbb X\times \dots \times \mathbb X$ denotes the joint feasible set of $\bm x$,  $\eta$ is the stepsize, the initial value $\bm x(0)$ is given. 
The gradient $\nabla \mathsf{C_T}$ can be evaluated by the partial gradient on each variable $x_t$:
%The partial gradient of $\mathsf{C_T}$ with respect to stage decision variable $x_t$ and can be evaluated by
$$\frac{\partial \mathsf{C_T}}{\partial x_t}(\bm x) = \nabla f_t(x_t)+ \beta(x_t-x_{t-1})+ \beta(x_t-x_{t+1})$$
when $1\leq t \leq T-1$, and $\frac{\partial \mathsf{C_T}}{\partial x_T}(\bm x) = \nabla f_T(x_T)+ \beta(x_T-x_{T-1})$ at stage $T$.
 {Notice that the partial gradient $\frac{\partial \mathsf{C_T}}{\partial x_t}(\bm x)$ only depends on the cost function $f_t$ and the stage variables $x_{t-1}, x_t, x_{t+1}$. To emphasize this fact, we slightly abuse the notation and write the partial gradient as $\frac{\partial \mathsf{C_T}}{\partial x_t}(x_{t-1:t+1})$. With this notation, the projected gradient descent \eqref{equ: gd x} can be written equivalently as follows. For iteration $k=1,2,\dots,$ the updating rule of GD on the stage  variable $x_t$ for $1\leq t \leq T$ is
\begin{equation}\label{equ: gd on xt}
x_t(k)=\Pi_{\mathbb X} \left[ x_t(k-1)-\eta \frac{\partial \mathsf{C_T}}{\partial x_t}(x_{t-1:t+1}(k-1))\right]
\end{equation}}
  {Rule \eqref{equ: gd on xt} shows that, to compute $x_t(k)$ by the offline gradient descent, we only need $f_t$ and $x_{t-1}(k-1), x_t(k-1), x_{t+1}(k-1)$, instead of all the cost functions and all the stage variables. This suggests that it is still possible to implement \eqref{equ: gd on xt} for a few iterations using only the finite lookahead window of cost functions. This is the key insight that motivates our online algorithm design below.}

\subsection{Receding Horizon Gradient Descent (RHGD)}
Inspired by the offline gradient descent, we  design our online  RHGD  (see  Algorithm \ref{alg:RHGD}).  {For ease of notation, we define $f_t(\cdot):=0$ for $t\leq 0$ or $t>T$ and let $x_t(k):=x_0$ for $t\leq 0$ and   $k\geq 0$ when necessary. At stage $t=1-W$, RHGD sets $x_1(0)=x_0$. At  stage $2-W\leq t\leq T$,  RHGD receives $f_t,\dots, f_{t+W-1}$ and runs the following two steps.}

 {\nbf{In Step 1,} RHGD initializes the  variable    $x_{t+W}(0)$ with an initialization method $\phi$. Notice that $\phi$ can be any  method that only uses the available information at $t$, i.e. $f_1,\dots, f_{t+W-1}$ and the stage variables computed before $t$. For instance, $\phi$ can be online gradient descent (OGD), which is a well-known  OCO algorithm in literature \cite{hazan2016introduction} and is provided below.
\begin{equation}\label{equ: ogd}
 x_{t+W}(0)= \Pi_{\mathbb X}
 \left[ x_{t+W-1}(0)- \gamma \nabla f_{t+W-1}(x_{t+W-1}(0))\right],
\end{equation}where $\gamma>0$ is the stepsize and $x_{t+W-1}(0)$ is available from the Step 1  at the previous stage $t-1$.}

% Since $f_{t+W}$ is unknown at stage $t$,  we adopt a well-known OCO algorithm in literature: online gradient descent (OGD) \cite{hazan2016introduction}:%the initialization scheme should not use $f_{t+W}$. One possible initialization scheme is  online gradient descent (OGD),  which is a widely used OCO algorithm in  literature \cite{hazan2016introduction}.  OGD is provided below.
% %Since $f_{t+W}$ is unknown, $x_{t+W}(0)$ can be initialized by OCO algorithms, e.g. online gradient descent (OGD),  which is a widely used online algorithm in  literature \cite{hazan2016introduction}:
% \begin{equation}\label{equ: ogd}
%  x_{t+W}(0)= \Pi_{\mathbb X}
%  \left[ x_{t+W-1}(0)- \gamma \nabla f_{t+W-1}(x_{t+W-1}(0))\right],
% \end{equation}
% where $\gamma>0$ is the stepsize and $x_{t+W-1}(0)$ is available from the Step 1  at the previous stage $t-1$. We mention  that the  OGD initialization scheme %is not restricted to OGD  and 
% can be replaced by other  schemes that only require the available online  information.

\begin{algorithm}\caption{Receding Horizon Gradient Descent (RHGD)}
	\label{alg:RHGD}
	\begin{algorithmic}
		\STATE \textbf{Inputs:} $x_0$, $\mathbb X$, $\beta$, $W$, stepsizes $\gamma,\eta$
		\STATE Let $x_1(0)=x_0$. 
		%	\STATE For ease of notation, let $x_{t}(k)= x_0$ for $t\leq 0$ and $k\geq 0$, and $f_t(\cdot)=0$ for $t\leq 0$ and $t\geq T+1$.
		%	\STATE \textbf{Update}:
		\FOR{$t=2-W$ \textbf{to} $T$}
		\STATE  \texttt{Step 1: initialize  $x_{t+W}$ by OGD. } 
		
		%	\IF{$t+W\leq T$}
		
		\STATE$
		x_{t+W}(0)= \Pi_{\mathbb X}\left[x_{t+W-1}(0)-{\gamma}\nabla f_{t+W-1}(x_{t+W-1}(0)) \right]
		$
		%	\ENDIF
		
		\STATE  \texttt{Step 2: update  $x_{t+W-1},x_{t+W-2},\dots, x_t$.}
		\FOR{$s=t+W-1$ \textbf{downto} $t$}
		\STATE compute $x_s(k)$ by projected gradient descent \eqref{equ: gd on xt}, 
		\begin{align*}
		&x_s(k)= \Pi_{\mathbb X}\left[x_s(k-1)-\eta  \frac{\partial \mathsf{C_T}}{\partial x_s}(x_{s-1:s+1}(k-1))\right]
		\end{align*}
		where $k=t+W-s$.
		
		\ENDFOR
		\STATE \textbf{Output:} $x_t(W)$.
		\ENDFOR

	\end{algorithmic}
\end{algorithm}

 {\nbf{In Step 2,} RHGD updates the values of $x_{t+W-1}, x_{t+W-2}, $ $\dots, x_t$ one by one by  \eqref{equ: gd on xt}. In the following, we show that Step 2 computes the exact values of $x_{t+W-1}(1),$ $ \dots, x_t(W)$ defined in the offline gradient descent \eqref{equ: gd on xt} for the offine optimization \eqref{equ: soco} by only using the available online information at $t$.} %and computes the exact values of the variables $x_{t+W-1}(1), \dots, x_t(W)$ in the offline gradient descent iterations \eqref{equ: gd x}. 
 {\begin{itemize}[label={--}]
    \item At first, RHGD  computes $x_{t+W-1}(1)$ exactly by \eqref{equ: gd on xt} since $f_{t+W-1}$ is received, $x_{t+W}(0)$ is computed in Step 1,  $x_{t+W-1}(0)$ and $x_{t+W-2}(0)$ have been computed in Step 1 of the stages $t-1$ and  $t-2$ respectively.
    \item Next, RHGD   computes $x_{t+W-2}(2)$ exactly by \eqref{equ: gd on xt} since $f_{t+W-2}$ is received, $x_{t+W-1}(1)$ is computed above,  $x_{t+W-2}(1)$ and $x_{t+W-3}(1)$ have been computed in Step 2 of  the stages $t-1$ and $t-2$ respectively.
    \item Similarly, RHGD computes $x_{t+W-3}(3),$ $ \dots, x_{t}(W)$ one by one  based on the received  costs  and the values that have been computed at stages $t,t-1,t-2$. %   at the current stage and the previous two stages.
\end{itemize}
%At the beginning, RHGD can compute the value of $x_{t+W-1}(1)$ by \eqref{equ: gd on xt} because $f_{t+W-1}$ is received, $x_{t+W}(0)$ is computed in Step 1, and $x_{t+W-1}(0)$ and $x_{t+W-2}(0)$ have been computed in Step 1 of stage $t-1$ and $t-2$ respectively. Next, RHGD  computes $x_{t+W-2}(2)$ by \eqref{equ: gd on xt} because $f_{t+W-2}$ is received, $x_{t+W-1}(1)$ is computed above, and $x_{t+W-2}(1)$ and $x_{t+W-3}(1)$ have been computed in Step 2 of stage $t-1$ and $t-2$ respectively. Following the same logic, RHGD computes $x_{t+W-3}(3), \dots, x_{t}(W)$ one by one by using the received cost functions  and the variables computed at the current stage and the previous two stages. 
The final output at stage $t$ is $x_t(W)$, which is the same as that in the $W$th iteration of  the offline gradient descent \eqref{equ: gd on xt} for $\mathsf{C_T}$.}

 %for the total cost function $\mathsf{C_T}$  to update $x_{t+W-1}, x_{t+W-2}, \dots, x_t$ by  using only the $W$-lookahead information. This is implemented by (i) first updating $x_{t+W-1}(1)$ according to \eqref{equ: gd on xt}, which utilizes the predicted cost function $f_{t+W-1}$, the  initial  value $x_{t+W}(0)$ computed in Step 1, and the initial values $x_{t+W-1}(0)$ and $x_{t+W-2}(0)$ computed in previous stages; (ii) then updating $x_{t+W-2}(2)$ by \eqref{equ: gd on xt}, which utilizes the predicted cost $f_{t+W-2}$,  the  value $x_{t+W-1}(1)$ computed in (i) above, and $ x_{t+W-2}(1), x_{t+W-3}$ computed at previous stages; and so on. Eventually, we are able to compute and output $x_t(W)$. Notice that $x_t(W)$  is also the $W$th update of the offline gradient descent \eqref{equ: gd on xt} on $x_t$ for the total cost function $\mathsf{C_T}(\cdot)$. 

%It can verified that RHGD only uses the online information available at stage $t$, thus being an admissible online algorithm.  

%RHGD is summarized in Algorithm \ref{alg:RHGD}.  For ease of notation, we let $x_t(k)=x_0$ for $t\leq 0$ and $f_t(\cdot)=0$ for $t\leq 0$ or $t>T$. 

%It is straightforward  that the outputs of RHGD at each stage $x_1(W),\dots, x_T(W)$ constitute the output of projected gradient descent \eqref{equ: gd x} after $W$ iterations given the same initial values $x_1(0), \dots, x_T(0)$.

 {Notice that RHGD (with OGD  initialization) only requires $W+1$ projected gradient evaluations at each stage.} Therefore, our RHGD is computationally efficient when the projection onto $\mathbb X$ can be evaluated efficiently, e.g. when $\mathbb X$ is a positive orthant,
an n-dimensional box, a probability simplex, a Euclidean ball,
etc.
%, our RHGD at each stage is  computationally efficient.

%The computation required at each stage $t$ is mostly the $W$ projected gradient updates in Step 2 and the computation required by the initialization scheme in Step 1. If applying online gradient descent as initialization, the total number of projected gradient evaluations will be $W+1$. The projection onto set $X$ admits fast algorithms when set $X$ is e.g. a positive orthant,
%an n-dimensional box, a probability simplex, a Euclidean ball,
%etc. \blue{Thus, when the projection onto $X$ and gradient evaluations can be computed efficiently, our RHGD requires limited computation time per stage.}

%When the set X is simple, such as a positive orthant,
%an n-dimensional box, a probability simplex, a Euclidean ball,
%etc, the projection onto X admits fast algorithms. In this case, 

 %When the projection onto $X$ and gradient evaluations can be computed efficiently, our RHGD requires limited computation time per stage.

\begin{example}[Illustrative example]\label{example: illustration}

% {To revise}
		\begin{figure}[htpb]
		\centering
		\includegraphics[width=0.6\textwidth]{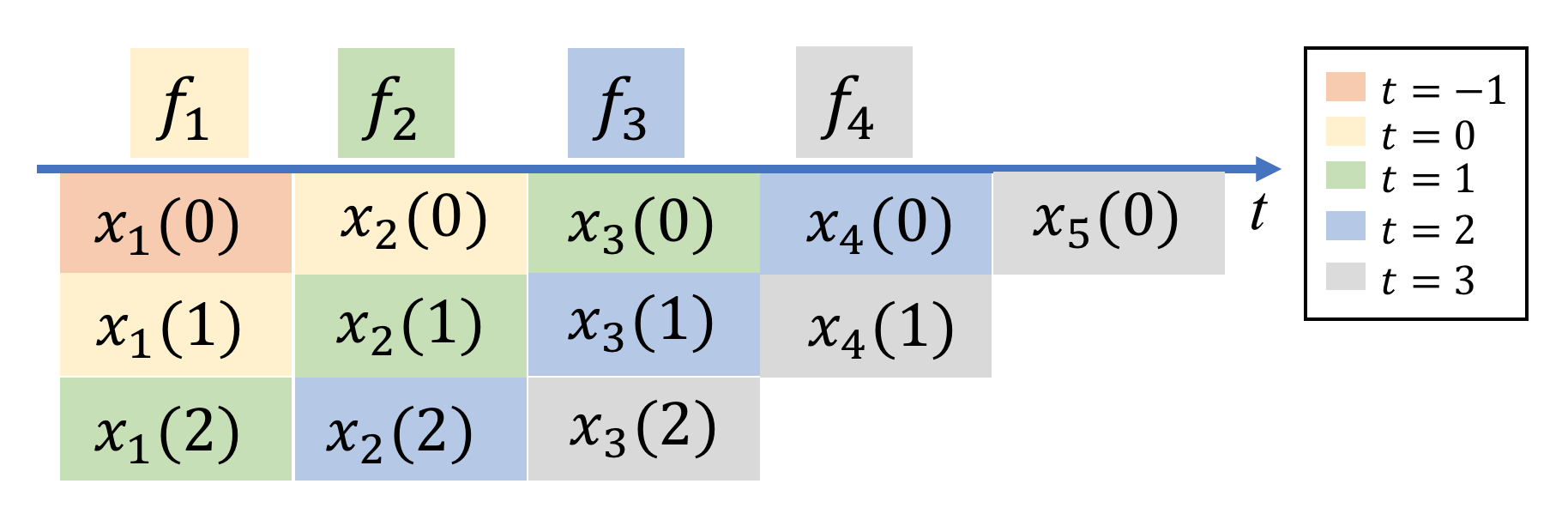}
			\caption{An example when $W=2$. The new cost function received and the new variables computed by RHGD  at stage $t=-1,0, 1,2,3$ are  marked in different colors.}% in yellow, green, blue boxes respectively. The orange box marks the computation before the problem starts.}
		\label{fig: illustration}
	\end{figure}
{Figure \ref{fig: illustration} provides an illustrative example for RHGD when $W=2$. Define $x_t(k)=x_0$ for $t\leq 0$, $k\geq 0$; $f_t=0$ for $t\leq 0$. At $t=-1$,  set $x_1(0)=x_0$. %At $t\geq 2-W=0$, the procedures of RHGD are discussed below.

%At $t=-1$, we let $x_1(0)=x_0$. % and its connection with the offline gradient descent. %For ease of notation,  we define $x_0(0)=x_0(1)=x_0$, $f_0(\cdot)=0$. 

%At $t=-1$, we let $x_1(0)=x_0$. 
\begin{itemize}
    \item At  $t=0$, RHGD receives $f_0$ and $f_1$. Step 1  initializes $x_2(0)$ by OGD \eqref{equ: ogd} with $f_1$ and  $x_1(0)$. Step 2 computes $x_1(1)$ by GD \eqref{equ: gd on xt} with $f_1$ and  $x_0(0), x_1(0), x_2(0)$; and then computes $x_0(2)$ for ease of notation, which is omitted in Fig. \ref{fig: illustration}.
    
    %For ease of notation, RHGD   updates $x_0(2)$, which is omitted in Fig. \ref{fig: illustration}.

\item At $t=1$, RHGD receives $f_1, f_2$. Step 1 initializes $x_3(0)$ by   \eqref{equ: ogd} with  $f_2$ and  $x_2(0)$. Step 2 first computes $x_2(1)$ by   \eqref{equ: gd on xt} with  $f_2$ and  $x_1(0), x_2(0), x_3(0)$ and then computes $x_1(2)$ by \eqref{equ: gd on xt} with  $f_1$ and  $x_0(1), x_1(1), x_2(1)$. RHGD  outputs $x_1(2)$. 

%At
% $t=2$, RHGD receives $f_2, f_3$, initializes $x_4(0)$ by \eqref{equ: ogd} with $f_3$, $x_3(0)$ in Step 1; computes $x_3(1)$ by \eqref{equ: gd on xt} with $f_3, x_4(0), x_3(0), x_2(0)$; then computes $x_2(2)$ by \eqref{equ: gd on xt} with $f_2, x_3(1),x_2(1), x_1(1)$; and outputs $x_2(2)$. 
 
 \item Similarly, at $t=2$, RHGD  receives $f_2, f_3$, initializes $x_4(0)$, computes $x_3(1)$  and then $x_2(2)$ by \eqref{equ: gd on xt}, and outputs $x_2(2)$. 
 \item Similarly, at $t=3$, RHGD initializes $x_5(0)$, computes $x_4(1)$ and then $x_3(2)$, and outputs $x_3(2)$. So on and so forth.%not necessary  and only included for ease of notation, thus omitted in Fig. \ref{fig: illustration}.
 
\end{itemize}
}
\end{example}

\vspace{-20pt}

\subsection{Receding Horizon Accelerated Gradient (RHAG)}

The design idea of RHGD can be extended to other gradient methods, e.g. Nesterov's accelerated gradient (NAG), triple momentum,  etc. Due to space limits, we only introduce RHAG (Algorithm \ref{alg:RHAG}) based on NAG. NAG updates $x_t(k)$ and  auxiliary variable $y_t(k)$ for $1\leq t \leq T$ at iteration $k$ by
\begin{equation}\label{alg:nag}
\begin{aligned}
& x_{t}{(k)}=\Pi_{\mathbb X}\left(y_t{(k-1)}-\eta \frac{\partial \mathsf {C_T}}{\partial x_t}(y_{t-1:t+1}(k-1)) \right),\\
& y_t{(k)}=(1+\lambda)x_t{(k)}-\lambda x_t{(k-1)}.
\end{aligned}
\end{equation}
% where the auxiliary variable $y$ provides momentum to accelerate the convergence, and $\eta$ and $\lambda$ are stepsizes.

 {Similar to RHGD, RHAG also conducts two steps at each  $t$. The differences are:  in Step 1, RHAG initializes not only $x_{t+W}(0)$ but also  $y_{t+W}(0)$; and in Step 2, RHAG updates $x_{t+W-1}(1),\dots, x_t(W)$ by  NAG \eqref{alg:nag}  instead of the  gradient descent. Nevertheless, RHAG still outputs $x_t(W)$, which is the value of $x_t$ after $W$th iterations of NAG.}  {In total, RHAG also only requires $W+1$ projected gradient evaluations.} %, which can be computed in a short time when the projection onto $X$ and the gradient evaluations can be computed efficiently.}

\begin{algorithm}\caption{Receding Horizon Accelerated Gradient (RHAG)}
		\label{alg:RHAG}
	\begin{algorithmic}
		\STATE \textbf{Inputs:} $x_0$, $\mathbb X$, $\beta$, $W$, stepsizes $\gamma, \eta, \lambda$.
		\STATE Let $x_1(0)=y_1(0)=x_0$.
		%	\STATE For ease of notation, let $x_{t}(k)=y_t(k)= x_0$ for $t\leq 0$ and $k\geq 0$, and $f_t(\cdot)=0$ for $t\leq 0$ and $t\geq T+1$.
%		\STATE Let $\eta = 1/L$, $\lambda= \frac{1-\sqrt {\alpha\eta}}{1+\sqrt{ \alpha\eta}}$
		%	\STATE \textbf{Update}:
		\FOR{$t=2-W$ \textbf{to} $T$}
		\STATE  \texttt{Step 1: initialize  $x_{t+W}, y_{t+W}$ by OGD.} 
	
		\STATE 	 $
		x_{t+W}(0)= \Pi_{\mathbb X}\left[x_{t+W-1}(0)-{\gamma}\nabla f_{t+W-1}(x_{t+W-1}(0)) \right]
		$ 
		\STATE Let $y_{t+W}(0)=x_{t+W}(0)$.

		\STATE  \texttt{Step 2: update $(x_{t+W-1}, y_{t+W-1}), \dots, (x_t, y_t)$.}
		\FOR{$s=t+W-1$ \textbf{downto} $ t$}
	\STATE	compute $(x_s(t+W-s), y_s(t+W-s))$ by NAG
		\begin{align*}
		&x_s(k)= \Pi_{\mathbb X}\left[y_s(k-1)-\eta \frac{\partial \mathsf{C_T}}{\partial x_s}(y_{s-1:s+1}(k))\right]\\
		&y_s(k)=(1+\lambda)x_s(k)-\lambda x_s(k-1)
		\end{align*}
		where $k=t+W-s$.
		\ENDFOR
		\STATE \textbf{Ouput}:  $x_t(W)$.
		\ENDFOR
		
	\end{algorithmic}
\end{algorithm}

\begin{remark}
	 {We note that RHGD and RHAG do not use the  latest information. To see this, consider Example \ref{example: illustration}. RHGD updates $x_3(1)$ by using $x_2(0)$  though $x_2(1)$ is available  (the same applies to RHAG). This allows RHGD (RHAG) to exactly implement offline GD (NAG) for $\mathsf{C_T}(\bm x)$, which greatly simplifies the  analysis and maintains the convergence rates of GD (NAG). Numerical results in our online report \cite{li2018online} show that RHGD and its variant that uses the latest information perform similarly, while RHAG  outperforms its variant, which is likely due to the sensitivity of NAG to disturbances \cite{devolder2014first}.}
\end{remark}	
\section{Theoretical Analysis of Our Algorithms}\label{sec: theoretical analysis}
% {This part is very different, but I may need to rewrite the proof, let me worry about it later.}
This section provides the regret upper bounds of 
   RHGD and RHAG, as well as our initialization method OGD. In addition, we provide the stability guarantees of RHGD and RHAG. % is discussed.

%and stability guarantees of our online algorithms RHGD and RHAG, as well as the classic OCO algorithm OGD.

%, before which we first present a supportive lemma on $\mathsf{\mathsf{C_T}}$'s  properties.

%Before the regret upper bounds, we first provide a supportive lemma on the strong convexity and the smoothness of $\mathsf{C_T}(\bm x)$. %The proof is deferred to \cite{li2018online}.

% {Should I mention OGD in this paper? I should, this is my theoretical results. I would rather delete some discussion to make space for the theorem. But, how to make it natural? I should add some explanation, e.g. OGD does not use prediction and serves as our initialization.  RHGD and RHAG uses OGD as initialization and improves its performance by using predictions. }

We first prove a supportive lemma on the properties of $\mathsf{C_T}$.

\begin{lemma}\label{lem:cost}
Given Assumption \ref{ass:1}, $\mathsf{C_T}(\bm x)$ defined in \eqref{equ: soco} is $\alpha$-strongly convex and  $L$-smooth on $\R^n$ for $L=l+4\beta$. %, thus the condition number being $Q_f= L/\alpha= \frac{l+4\beta}{\alpha}$.
	%\footnote{Here, $\alpha$, $l+4\beta$, and $\frac{l+4\beta}{\alpha}$ are merely estimates of the strong convexity and smoothness factors and the condition number, and can be replaced with better estimation.}
\end{lemma}
\begin{proof}
The Hessian of $\sum_{t=1}^T\frac{\beta}{2} \|x_t-x_{t-1}\|^2$ has eigenvalues in $[0, 4\beta]$ by the Gershgorin circle theorem. Thus, by  Assumption \ref{ass:1},  $\mathsf{C_T}(\bm x)$ is $\alpha$ strongly convex and $L=l+4\beta$ smooth.
\end{proof}

%\nit{Proof. }The Hessian of $\sum_{t=1}^T\frac{\beta}{2} \|x_t-x_{t-1}\|^2$ has eigenvalues in $[0, 4\beta]$ by the Gershgorin circle theorem. Thus, by  Assumption \ref{ass:1},  $\mathsf{C_T}(\bm x)$ is $\alpha$ strongly convex and $l+4\beta$ smooth. \qed

% Firstly, by Assumption \ref{ass:1}(i) and the convexity of $\sum_{t=1}^T \|x_t-x_{t-1}\|^2$,   $\mathsf{C_T}(\bm x)$ is $\alpha$ strongly convex. Secondly, the Hessian of $\sum_{t=1}^T \|x_t-x_{t-1}\|^2$ is a symmetric matrix with eigenvalues in $[0, 4\beta]$ by Gershgorin circle Theorem. Thus, by applying Assumption \ref{ass:1} (ii),  $\mathsf{C_T}(\bm x)$ is $L=l+4\beta$-smooth.

% {how about I remove OGD? but I still should include the proof in the appendix. so little use?}

%This subsection provides the regret bounds on RHGD and RHAG, as well as some discussion.
%The following are the upper bounds of OGD, RHGD, RHAG.

%Next, we provide the regret upper bounds of RHGD and RHAG, as well as OGD.

%The theorem below provides upper bounds on RHGD and RHAG's dynamic regrets.

Next, we provide the regret upper bounds.

\begin{theorem}[General regret upper bounds]\label{thm: regret upper bound general initial}  {Under Assumption \ref{ass:1},
 for $W\geq 0$, 
 given
 stepsizes  $\eta=1/L$ and $\lambda= \frac{1-\sqrt {\alpha/L}}{1+\sqrt{ \alpha/L}}$,\footnote{ {The stepsize conditions  can be relaxed as in [29],
yielding  different bounds.}} for any initialization $\phi$ in Step 1, we have}
  {\begin{align}
    & \textup{Reg}(RHGD)
	\leq Q_f  \left(1-\frac{1}{Q_f}\right)^{W} 	\textup{Reg}(\phi)\\
&	\textup{Reg}(RHAG)
	\leq  2  \left(1-\frac{1}{\sqrt{Q_f}}\right)^{W}	\textup{Reg}(\phi)
 \end{align}
 	where $Q_f=\frac{L}{\alpha}$ and $\textup{Reg}(\phi)$ is the regret of implementing the initial values $\{x_t(0)\}_{t=1}^T$ computed by the initialization $\phi$.}
 	\end{theorem}
 	\begin{proof}
 		Notice that  $x_t(W)$ is the $W$th iterate of GD for $\mathsf{C_T}(\bm x)$.  By GD's convergence rate (see Appendix B in \cite{li2018online}), we have
	\begin{align*}
	\text{Reg}(RHGD)&= \mathsf{C_T}(\bm x(W))-\mathsf{C_T}(\bm x^*)\\
	& \leq Q_f \left(1-1/Q_f\right)^W \left[\mathsf{C_T}(\bm x(0))-\mathsf{C_T}(\bm x^*)\right]\\
	& = Q_f \left(1-1/Q_f\right)^W	\text{Reg}(\phi)
	\end{align*}
	$\text{Reg}(RHAG)$ can be bounded similarly  by   NAG's convergence rate (see e.g. \cite{nesterov2013introductory}).
 	\end{proof}
 	
% 	 {add discussion on stepsize inequality}

\begin{theorem}[Regret upper bounds with OGD initialization]\label{thm: dynamic regret}  {When Assumption \ref{ass:1}-\ref{ass:2} hold,} the regret of the initialization method OGD \eqref{equ: ogd} with stepsize $\gamma=1/l$ is  bounded by,\footnote{ {Stepsizes $\gamma \leq 1/l$ also work, yielding  different constant factors.}}
\begin{equation}
    	\textup{Reg}(OGD) \leq \delta \sum_{t=1}^T\|\theta_t-\theta_{t-1}\|,\label{equ: ogd bdd}
\end{equation}
	where 
		$\delta= (\beta/l+1)\frac{G}{(1-\kappa)}$, $\kappa=\sqrt{(1-\frac{\alpha}{l})}$, $\theta_t=\argmin_{x\in \mathbb X} f_t(x)$ for $1\leq t \leq T$, and $\theta_0=x_0$. 

Consequently, the regrets of RHGD and RHAG with OGD initialization are upper bounded by 
	\begin{align}
	&	\textup{Reg}(RHGD)
	\leq  \delta Q_f  \left(1-\frac{1}{Q_f}\right)^{W} \sum_{t=1}^T\|\theta_t-\theta_{t-1}\|,\label{equ: rhgd bdd} \\
	&	\textup{Reg}(RHAG)
	\leq  2 \delta \left(1-\frac{1}{\sqrt{Q_f}}\right)^{W}\sum_{t=1}^T\|\theta_t-\theta_{t-1}\|.\label{equ: rhag bdd}
	\end{align}
\end{theorem}
The proof of Theorem \ref{thm: dynamic regret} is deferred to Appendix 1.

\nbf{Effect of $W$.} When $W=0$,  RHGD and RHAG reduce to OGD. When $W\geq 0$, RHGD and RHAG's  regret bounds  exponentially decay  with $W$. RHAG's bound decays faster than RHGD's since NAG converges faster than GD.

% {revise}

\nbf{Path length.} The  regret  bounds in Theorem \ref{thm: dynamic regret} all linearly depend on $\sum_{t=1}^T\|\theta_t-\theta_{t-1}\|$. OGD's bound  is similar to the bound for OCO without switching costs \cite{mokhtari2016online}.  {The term $\sum_{t=1}^T\|\theta_t-\theta_{t-1}\|$, called the \textit{path length} in literature \cite{mokhtari2016online}, consists of  (i) $\sum_{t=2}^T\|\theta_t-\theta_{t-1}\|$ that captures the variation  of the  cost functions $\{f_t\}_{t=1}^T$,  (ii) $\|\theta_1-\theta_0\|$ that represents  the quality of the initial point $x_0$ since $\theta_0=x_0$. When (i)  is close to zero, the path length is dominated by (ii). Otherwise, the path length is dominated by the variation of the cost functions.}

 %When the cost variation is large, the path length is dominated by the variation term (i). When the cost variation is close to zero, the path length is dominated by the quality of the initial point (ii), which is consistent with time-invariant optimization. 

% When the cost function changes slowly, the regret bound is small, which is intuitive since it is easy to follow a slowly changing cost. When cost function $f_t$ is  not changing with stage $t$, the regret is a constant depending on   $\|x_0-\argmin_{X} f_1(x)\|$, which is consistent with our intuition from classic optimization. When cost functions change drastically, the regret can be linear in $T$, which has been observed in literature \cite{besbes2015non,mokhtari2016online}. 

%\nbf{Initialization.}
%The regret bounds \eqref{equ: rhgd bdd} and \eqref{equ: rhag bdd} still hold with other initialization methods by replacing $\text{Reg}(OGD)$ with the regret of the corresponding initialization method. 

%\vspace{5pt}

%\nbf{Asymptotic stability.} We have the following stability result.

% {Note: OGD also stabilize the system, maybe I should mention it here.}
Finally, we provide some stability results of our algorithms. %More discussion on the stability is left as future work. % of RHGD and RHAG.

{\begin{corollary}[Asymptotic stability]\label{cor: asym stability}
Consider an  optimal control problem that is equivalent to  problem \eqref{equ: soco} when $T= \infty$:
	\begin{align}\label{equ: opt. control}
	\min_{x_t \in \mathbb X}&\ \sum_{t=1}^{+\infty} \left[f_t(x_t)+\frac{\beta}{2}\|u_t\|^2\right], \text{s.t. } x_t=x_{t-1}+u_t, \, \forall t.%\footnote{We note that the optimal control problem satisfies that the optimal steady state is the optimal solution to the stage cost solution, while EMPC usually consider scenarios that do not satisfy this condition. Nevertheless, our regret analysis focuses on the transient optimality}
	\end{align}
	Notice that $\theta_t:=\argmin_{\mathbb X} f_t(x)$ is the optimal steady state for the stage cost at  $t$. 
	If $\theta_t \to \theta_{\infty}$ as $t\to +\infty$ and 
	 $\sum_{t=1}^{+\infty} \|\theta_t -\theta_{t-1}\| <+\infty$,  then $x_t(W)\to \theta_{\infty}$ as $t\to +\infty$, where $x_t(W)$ denotes the output of RHGD (or RHAG) and $W\geq 0$.
	 
	 %RHGD and RHAG yield asymptotically stable solutions: $\|x_t(W)-\theta_t\|\to0$ as $t\to +\infty$, where $x_t(W)$ denotes the output of RHGD or RHAG

% 		Consider the following optimal control problem that is equivalent to our online optimization problem with switching costs:
% 	\begin{align*}
% 	\min_{x_t \in \mathbb X}&\quad \sum_{t=1}^T f_t(x_t)+\frac{\beta}{2}\|u_t\|^2\\
% 	\text{s.t. } & \quad x_t=x_{t-1}+u_t, \quad \forall\ t\geq 1
% 	\end{align*}
% 	When the path length is bounded $\sum_{t=1}^{+\infty} \|\theta_t -\theta_{t-1}\| <+\infty$, the outputs of our RHGD and RHAG are asymptotically stable in the sense that $\|x_t(W)-\theta_t\|\to0$ as $t\to +\infty$, where $x_t(W)$ can be the output of RHGD or RHAG and $\theta_t$ is the optimal steady state with the stage cost function at $t$.
\end{corollary}}

\section{Fundamental Limits}\label{sec: lower bound}
This section provides fundamental lower bounds on the
dynamic regrets for both $W = 0$ and $W > 0$. For simplicity,
we only consider deterministic online algorithms in this paper.

%This section provides fundamental limits of any admissible online algorithms on the dynamic regrets  for both $W=0$  and  $W>0$.  {An algorithm is admissible if it only uses the available information at each stage $t$, i.e. $\{f_1, \dots,  f_{t+W-1}\}$.} For simplicity, we only consider deterministic algorithms here.

%The admissible online algorithms refer to the algorithms that only use the available online information at each stage $t$,  i.e. $\{f_1, \dots,  f_{t+W-1}\}$.
%For  simplicity, we only consider deterministic online  algorithms in this paper.
%In the following, we provide fundamental lower bounds on dynamic regrets of online algorithms for our SOCO problem \eqref{equ: soco} for both $W=0$ case (no prediction) and $W>0$ case (with predictions). For the technical simplicity, we consider online deterministic algorithms in this section.%\footnote{ {Rigorously speaking, we also need $\A(\cdot)$ to be a measurable function on some measurable space containing $\mathds I_t$. Most existing algorithms satisfy this condition. This is discussed more in \cite{li2018online}.}}

%\vspace{6pt}

%\noindent\textcolor{nblue}{\normalfont\textit{i) The case with no prediction}}
%\vspace{-6pt}

\begin{theorem}[No prediction]\label{thm: W=0 lower bound}
	Consider  $W=0$. Given any 
	$T\geq 0$,  $\alpha>0$, $\beta\geq  0$,  {a convex compact set $\mathbb X$ with diameter $D>0$,} and  $0\leq L_T\leq DT$, for any deterministic  online  algorithm $\A$,  {there  exist a sequence of  quadratic functions $\{f_t(\cdot)\}_{t=1}^T $} with $\alpha$ strong convexity and $\alpha$ smoothness   on $\R^n$, gradient bound $(3\alpha + \beta)D$ on $\mathbb X$,  path length $\sum_{t=1}^T \|\theta_t-\theta_{t-1}\|\leq L_T$, where $\theta_t=\argmin_{\mathbb X}f_t(x)$, such that the regret is lower bounded by 
	\begin{align}\label{equ: W = 0 lower bound}
	\textup{Reg}(\mathcal A)\geq\zeta D  L_T\geq \zeta D \sum_{t=1}^T \|\theta_t-\theta_{t-1}\|
	\end{align}
	where  $\zeta =\frac{\alpha^3 (1-\rho)^2 }{32(\alpha +\beta)^2}$,  $\rho =  \frac{\sqrt Q_f -1}{ \sqrt Q_f +1}$, and $Q_f=\frac{\alpha+4\beta}{\alpha}$.
\end{theorem}

The proof is provided in \cite{li2018online}. Next, we provide some discussion. Notice that $L_T$ is an upper bound on the path length $\sum_{t=1}^T\|\theta_t-\theta_{t-1}\|$. Theorem \ref{thm: W=0 lower bound} shows that the regret lower bound is linear with the path length when $W=0$, which matches OGD's upper bound \eqref{equ: ogd bdd}. Besides, Theorem \ref{thm: W=0 lower bound} shows that when $L_T=\Omega(T)$, no online algorithm can achieve $o(T)$ dynamic regret, which is similar to the claims in \cite{besbes2015non}.

\begin{theorem}[$W$-stage predictions]\label{thm: W>=1 lower bound}
	Given any $T>1$, $1\leq W \leq T/2$, 
	$\alpha>0$, $\beta\geq  0$,  {a convex compact set $\mathbb X$ with diameter $D>0$,}   $0\leq L_T\leq DT$, for any deterministic online  algorithm $\A$,  {there  exist a sequence of quadratic functions $\{f_t(\cdot)\}_{t=1}^T $}  with $\alpha$ strong convexity and $\alpha$ smoothness   on $\R^n$, gradient bound $\alpha D$ on $\mathbb X$,    path length $\sum_{t=1}^T \|\theta_t-\theta_{t-1}\|\leq L_T$, where $\theta_t=\argmin_{\mathbb X}f_t(x)$, such that the regret satisfies

	\begin{equation}\label{equ: lower bound}
	\begin{aligned}
	&\textup{Reg}(\mathcal A)\geq
	\begin{cases}
	\frac{\zeta  D}{3}\rho^{2W}L_T, & \text{if $ L_T \geq D$,}\\
	%	\frac{\alpha D }{96} (1-\rho)^2 \left(\frac{\alpha}{\alpha+ \beta}\right)^2
	\frac{\zeta }{3}\rho^{2W}L^2_T,
	& \text{if $L_T <D$}
	\end{cases} 
	\end{aligned}
	\end{equation}
	where  $\rho =  \frac{\sqrt Q_f -1}{ \sqrt Q_f +1}$, $Q_f=\frac{\alpha+4\beta}{\alpha}$, and $\zeta =\frac{\alpha^3 (1-\rho)^2}{32(\alpha +\beta)^2}$ .
\end{theorem}

The main proof ideas are provided in Appendix 3. %Some discussion is provided below.

%Theorem \ref{thm: W=0 lower bound} and \ref{thm: W>=1 lower bound}'s proofs  are  in \cite{li2018online} and Appendix 2. % the proof of Theorem \ref{thm: W>=1 lower bound} is deferred to Appendix 2 and \cite{li2018online}. 

%Both theorems impose no constraints on the computational power of the online algorithms. Besides,
%both theorems only consider \textit{special/worst cases}, so the regret in practice  may be smaller than the provided lower bounds. 

\nbf{Special-case analysis.}  Notice that  Theorem   \ref{thm: W=0 lower bound} and \ref{thm: W>=1 lower bound}  only prove the fundamental lower bounds for  the \textit{special/worst cases}. It is possible to achieve smaller regrets in non-worst cases.

 %for som since claims that there exists some cost functions such that the lower bound hold, i.e. the theorem only considers  \textit{special/worst cases}. In some cases, the regret may be smaller than the lower bound. 

%meaning that the results and implication 

%the lower bounds hold the existence of some cost functions, 

\nbf{Comparion with RHAG.} Theorem \ref{thm: W>=1 lower bound} shows that, given limited predictions ($W\leq T/2$), the decay rate of the fundamental lower bound with respect to $W$ is similar to that of RHAG in Theorem \ref{thm: dynamic regret} for large $Q_f$, because to reach the same regret value $R$, the lower bound requires at least $W\geq \Omega((\sqrt{Q_f}-1) \log(L_T/R))$ (by $ \rho^{2W} \geq \exp(- \frac{4W}{\sqrt{Q_f}-1})$), while RHAG requires at most
$
W\leq O(\sqrt{Q_f} \log(L_T/R))
$ (by $(1- 1/\sqrt{Q_f})^W \leq\exp(-W/\sqrt{Q_f})$). In addition, given a reasonably large path length ($L_T \geq D$),  the  lower bound is linear with the path length, which matches the upper bound of  RHAG \eqref{equ: rhag bdd}. The comparison above indicates the near-optimality of RHAG at least in the special/worst cases, which may be surprising since RHAG only requires $W+1$ gradient evaluations while online algorithms in Theorem \ref{thm: W>=1 lower bound} have no computation constraints.

\nbf{More discussion on $L_T$.} Interestingly, when $L_T<D$, the lower bound in Theorem \ref{thm: W>=1 lower bound} is quadratic on $L_T$ which is smaller than the linear dependence as $L_T \to 0$. This is probably because with a small $L_T$, the online problem is ``easy'' and admits algorithms with better regrets.   When $W=0$, the  bound is the same for $L_T<D$, which is probably because the online algorithms have too limited future information  when $W=0$  to benefit from a small $L_T$. More discussion is in \cite{li2018online}.

\section{Numerical Experiments}
This section provides numerical results to complement our theoretical analysis by comparing our algorithms with MPC. 
%Some numerical results  are provided to complement the theory above. 

\begin{table}
	%	\captionsetup{justification=centering, labelsep=newline}
	\centering
	\begin{tabular}{c|c|c}
		\hline
		%	\diagbox{Algo.}{$W$} 
		Algorithm	& $W=5$      & $W=10$     \\ \hline
		RHGD & 3.48$ \times  10^{-4}$& 6.91$ \times  10^{-4}$ \\ %\hline
		RHAG &5.44$ \times  10^{-4}$ &  8.92$ \times  10^{-4}$\\ %\hline
		MPC  &  1.67$\times 10^{-1}$  & 3.31$\times 10^{-1}$ \\ \hline
	\end{tabular}
	\caption{Algorithms' running time  per stage in seconds.}
	\label{my-label}
\end{table}

% {Q: what 's the unit of the running time? how to show it in the table?}
\begin{figure}
\centering
	\begin{subfigure}[b]{0.4\textwidth}
		\includegraphics[width=\textwidth]{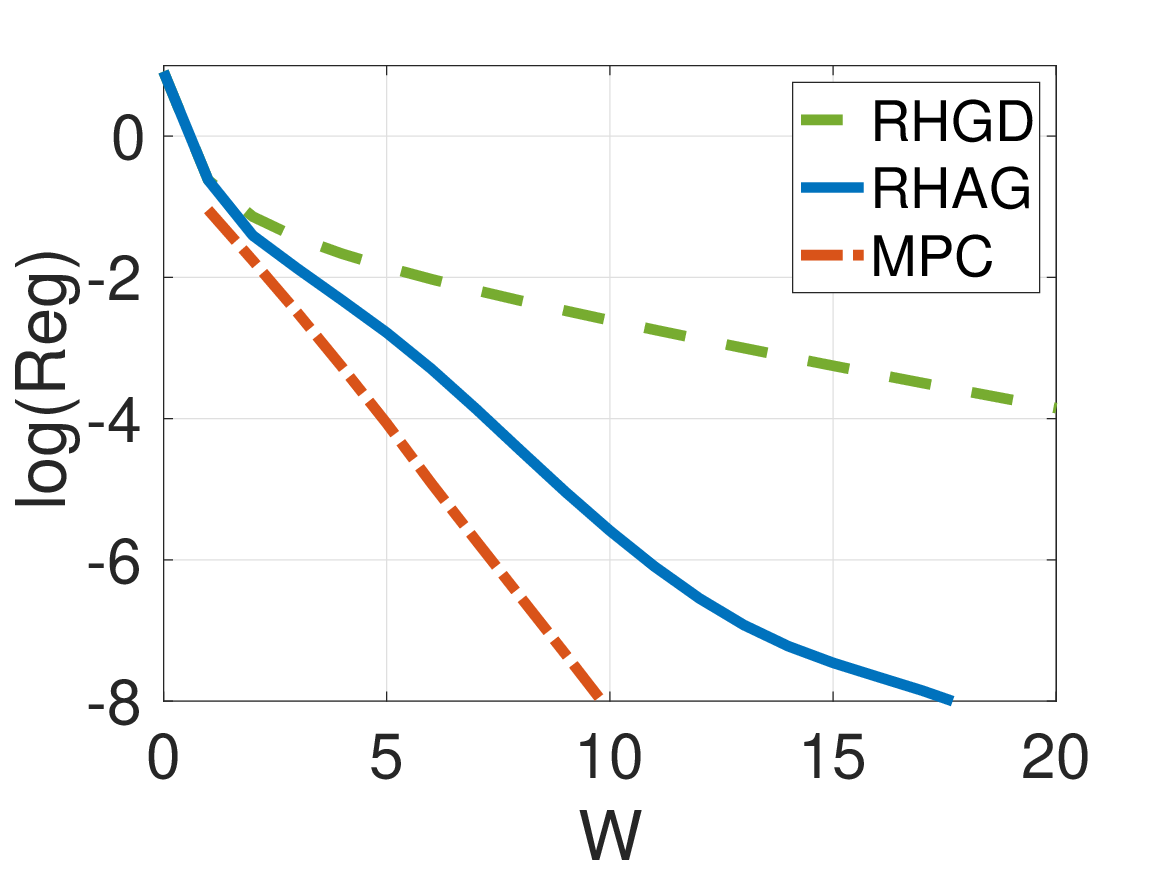}
		\caption{Smoothed regression}
		\label{fig: logistic}
	\end{subfigure} 
	\hfill
	\begin{subfigure}[b]{0.4\textwidth}
		\includegraphics[width=\textwidth]{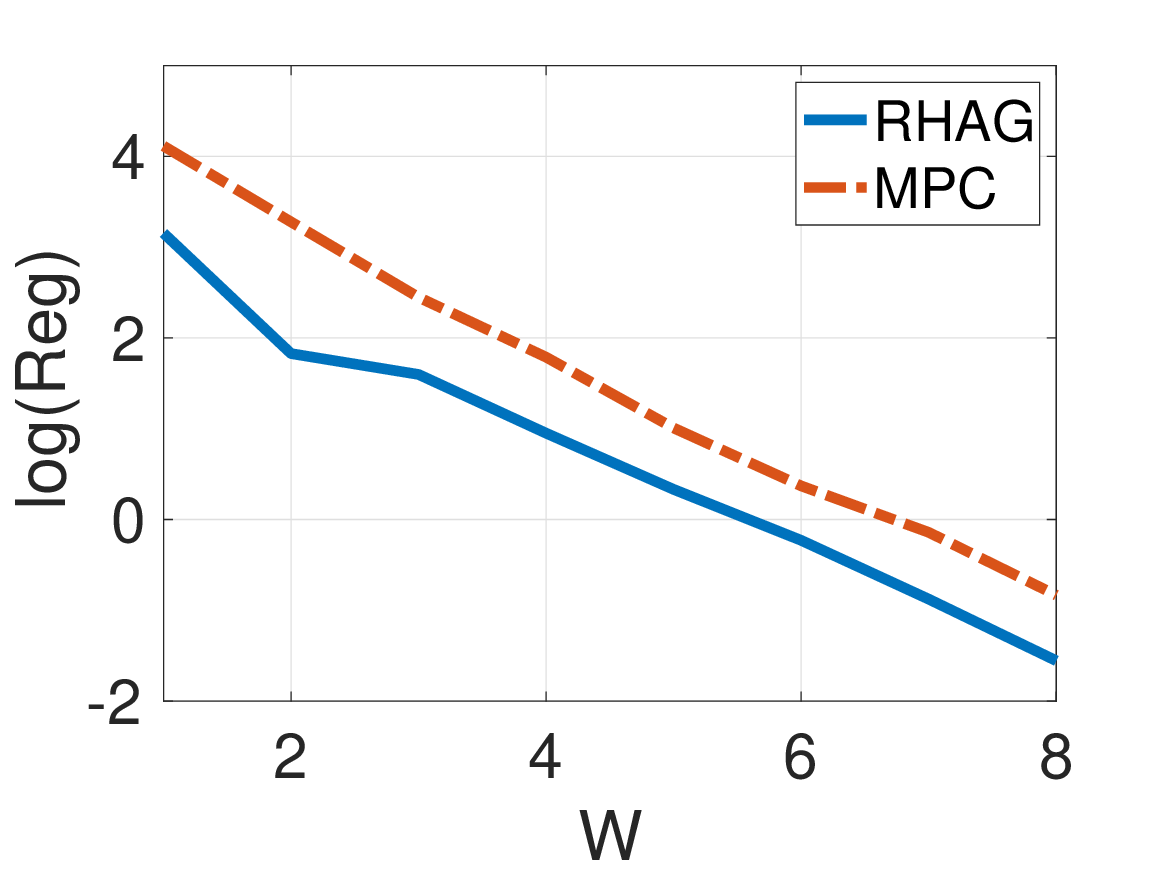}
		\caption{A special example}
		\label{fig:special}
	\end{subfigure}
	\caption{The regret of RHGD, RHAG, and MPC with   $W$. }
	\label{fig:regret}
\end{figure}
%\vspace{-10pt}

% {check parameter conflict.}
\subsubsection{Model predictive control}
%\textit{\underline{1Model Predictive Control}}%MPC is an optimization-based algorithm with good numerical performance. %We will compare our algorithm with a classic algorithm in the control literature, Model Predictive Control (MPC), which is known to enjoy good numerical performance \cite{rawlings2012postface}. 
We consider the following MPC algorithm. At each stage $t$, MPC solves the following open-loop $W$-stage receding horizon optimization
\begin{align}\label{equ: mpc}
\min_{x_{t}\dots x_{t+W-1}\in \mathbb X}\sum_{\tau=t}^{t+W-1} \left(f_\tau(x_\tau)+\frac{\beta}{2}\|x_\tau- x_{\tau-1}\|^2\right),
\end{align}
obtains the optimizer $\{x_t^t,\dots, x_{t+W-1}^t\}$, and implements  $x_t^t$.  {Though there are many variants of MPC and several methods to exploit the structures of MPC to speed up its computation, we limit ourselves to solving (\ref{equ: mpc}) by the iterative algorithm NAG. Note that at each iteration, NAG requires $W$ gradient evaluations since \eqref{equ: mpc} is a $W$-stage optimization problem. We terminate  NAG when the  gradient's norm is less than $10^{-14}$.}

%The iteration of NAG is terminated when the decrease of the objective cost in the current iteration is less than $10^{-14}$.}

%We set the termination rule to be ....

%\lina{Though there have been many variants of MPC and different methods to exploit the structure property of MPC for speeding up its computation, we limit ourselves on solving (\ref{equ: mpc}) through the iterative algorithm NAG. Note that at each iteration, NAG requires $W$ gradient evaluations because (\ref{equ: mpc}) is a $W-$stage optimization problem. We set the termination rule to be ....  } %Notice that MPC requires 

% {revise notation, avoid using $u$, and try to use a greek letter for $r$.}
\subsubsection{Smoothed regression}
% { Further, we numerically compare our gradient-based algorithms with an optimization-based MPC algorithm in two cases to complement our theoretical analysis. We show that our RHAG is comparable with MPC in a special (worst) case, which is aligned with our worst-case theoretical results. In addition, we show that RHAG  performs moderately worse than MPC in a more practical case but consumes much less computation time. This indicates that, in general (non-worst) cases, employing more computation may improve the  performance in practice. The theoretical analysis  beyond the worst case  is left as future work. }
 {Consider Example \ref{example: regression} with a loss function for logistic regression in \cite{logisticregression} with an $l_2$ regularizer:
$ f_t(x_t)= \frac{1}{M}\sum_{m=1}^M[ \log( 1+ e^{ w_{t,m}^\top x_t })-v_{t,m} w_{t,m}^\top x_t]+\frac{r}{2}\|x_t\|^2$,
where $M$ is the number of samples, $w_{t,m} \in \R^n$ are the features, $v_{t,m} \in \{0,1\}$ are the labels, $r$ is the $l_2$ regularization parameter. 
In particular, we consider $M=60$, $n=3$, $T=60$, $\beta=5$, and $r=0.5$. We generate $\{w_{t,m}\}_{m=1}^M$ as i.i.d. random Gaussian vectors with mean $\mu_t\mathbf 1_n$ and covariance $\sigma_t^2 \bm I_n$ for each $t$.  $\{\mu_t\}_{t=1}^T$ are i.i.d. from $\text{Unif}[-1,1]$, and  $\{\sigma^2_t\}_{t=1}^T$ are i.i.d.  from $\text{Unif}[0,1]$. We generate $\{v_{t,m}\}_{m=1}^M$  i.i.d.   from $\text{Bern}(p_t)$, where $\{p_t\}_{t=1}^T$ are i.i.d. from $\text{Unif}[0,1]$. Let $x_0=0$,  $\mathbb X=[-1,1]^3$ and choose stepsizes by Theorem \ref{thm: regret upper bound general initial} and \ref{thm: dynamic regret}.}   % {stepsize are not exactly the factors.} 
%The stepsizes are chosen according to Theorem \ref{thm: dynamic regret}.

Figure \ref{fig:regret}(a) plots the regrets  of RHGD and RHAG with OGD initialization and MPC  against $W$. Firstly,  the regrets of these algorithms decay linearly on a log scale, indicating the exponential decay with $W$. Further, RHAG decays faster than RHGD, which is consistent with Theorem \ref{thm: dynamic regret}.  {Besides, there are some fluctuations in  RHAG's regret caused by the fluctuations of NAG \cite{su2014differential}, but RHAG still yields smaller regrets than RHGD.} In addition,  Figure \ref{fig:regret}(a)  shows that MPC enjoys a  better decay rate with $W$. 
%Nevertheless,  to reach the same regret, the lookahead window $W$ needed by RHAG is roughly twice the window needed by MPC, indicating that RHAG exploits the lookahead information in a comparably efficient way. 
More analysis on MPC's regret  is left as future work. Table I compares the running time per stage of RHGD, RHAG, and MPC for $W=5, 10$.   {Notice that RHGD and RHAG are much faster than the MPC described in Section VI.1. This is because  our algorithms only evaluate $W+1$ gradients per stage, while we exploit multiple NAG iterations for solving MPC per stage and each iteration requires $W$ gradient evaluations. Nevertheless, we admit that there are other methods and variants of MPC that may further reduce the computation time, which is beyond the scope of this paper.}

\subsubsection{A special example}   {As indicated in Section \ref{sec: lower bound}, RHAG is near-optimal among all the deterministic online algorithms at least in the sense of the worst cases. We construct a special case to partially validate this indication by showing that RHAG slightly outperforms MPC in this case.} 
In particular, consider $ f_t(x_t)= 0.5(x_t-\theta_t)^2$,  $\mathbb X=[0,6]$, $T=20$. Let  $\{\theta_t\}_{t=1}^T$  be $[6,   0,    6  ,   0 ,    6,     6    , 0  ,  6  ,   6,0 ,    6  ,   6 ,    $ $ 0   ,  6 ,    6 ,    6  ,   6  ,   6,6,6]$. Let   $\beta =20$ and $\gamma=0.4$ and choose $\eta$ and $\lambda$  by Theorem \ref{thm: dynamic regret}.
Figure \ref{fig:regret}(b) shows that 
our gradient-based RHAG achieves slightly better performance than the optimization-based MPC in this case.

{\section{Conclusion}}

% {add non worst case}

 {This paper attempts to understand the fundamental value of the predictions and to study how to use predictions effectively in the online decision making problems} by considering an online convex optimization problem with a $W$-lookahead window on the cost functions and switching costs on the decisions. In this paper, we propose gradient-based online algorithms RHGD and RHAG and provide the dynamic regret upper bounds as well as some stability guarantees. We also characterize the fundamental limits of the dynamic regrets of any online algorithms for $W\geq 0$.  The fundamental lower bound when $W>0$ is close to RHAG's regret upper bound under certain conditions.   {There are many interesting  future directions, e.g. i) the algorithm design and analysis when there are errors and noises in the implementation, observations, and/or the  predictions of the future, in which cases RHGD may be better than RHAG  since NAG is sensitive to disturbances; ii) online algorithms when there are coupling constraints across stages and for more general dynamical systems; iii)  more analysis on the regrets of MPC and its variants; iv)  relaxing the    assumptions; v) customizing the lower bounds for different  function classes; vi)  beyond the worst-case analysis; vii) competitive ratio analysis; viii) algorithm design that involves generating predictions; ix) more discussion on stability.}

\appendix

% {Overview of supplemtary}

In the following, we provide more discussion for Theorem \ref{thm: W=0 lower bound} and Theorem \ref{thm: W>=1 lower bound} in Appendix \ref{aped: discussion on LT}, then provide a convergence rate analysis for projected gradient descent for ease of reference which is used in the proof of Theorem \ref{thm: regret upper bound general initial}. Moreover, we provide a proof for OGD's regret upper bound \eqref{equ: ogd bdd} (see Theorem \ref{thm: dynamic regret}) in Appendix \ref{append: ogd proof}. In addition,  we prove Corollary \ref{cor: asym stability} in Appendix \ref{aped: coro proof}. Next, we provide a proof of Theorem \ref{thm: W>=1 lower bound}  in Appendix \ref{append: proof of Thm 3 W>=1 LT>=D}, where a supportive lemma is proved in Appendix \ref{apen: x*=Atheta}. Besides, we provide a proof of Theorem \ref{thm: W=0 lower bound} in Appendix \ref{apen: W=0 lower bound}. Finally, we provide numerical results as promised in Remark 1 in Appendix \ref{aped: rhgd recent}.

\subsection{More discussion on Section V}\label{aped: discussion on LT}

% {where is OGD case?}

We provide more discussion on the case with a small $L_T$   in Theorem \ref{thm: W=0 lower bound} and Theorem \ref{thm: W>=1 lower bound}. When $L_T$ is small, i.e. $ L_T \leq D$, the lower bound is  $\Omega(L_T^2)$ as $L_T\to 0$, which is smaller than $\Omega(L_T)$ because $L_T^2 \leq DL_T$. An $O(L_T^2)$ upper bound can be achieved by a simple online algorithm:  $x_t^{\A}=\theta_t $. This is verified by the following arguments. Since $\theta_t$  minimizes each $f_t(\cdot)$,  the dynamic regret of  $x_t^{\A}=\theta_t $ is upper bounded by the switching costs, i.e. $\sum_{t=1}^T \frac{\beta}{2} \| \theta_t-\theta_{t-1}\|^2 $, and we have $\sum_{t=1}^T \| \theta_t-\theta_{t-1}\|^2\leq L_T^2 $. 
However, when there is no prediction, i.e., $W=0$, this simple online algorithm can not be implemented because $f_t(\cdot)$ is not available at stage $t$. In fact, the fundamental limit in Theorem \ref{thm: W=0 lower bound} indicates that no online algorithm can achieve $O(L_T^2)$ regret as $L_T\to 0$.

\subsection{Convergence rate of projected gradient descent}
The convergence rate analysis for projected gradient descent is standard in literature, but the common guarantee is on the distance between the decision variable and the optimal variable instead of on the gap between the algorithm's cost and the optimal cost (see e.g. \cite{nesterov2013introductory}). Since the proof of Theorem \ref{thm: regret upper bound general initial} requires the convergence rate in terms of the function value,  for ease of the reader and for the completeness of the paper, we provide a short proof for the convergence rate in terms of the function value below.

\begin{proposition}[Convergence rate of projected gradient descent in terms of the function value]
Consider constrained optimization $\min_{\mathbb Q} f(x)$ where $f(x)$ is $\alpha$ strongly convex and $L$ smooth and $\mathbb Q$ is convex. Then, the projected gradient descent on $f(x)$, i.e. $x(k+1)=\Pi_{\mathbb Q}[x(k)-\frac{1}{L}\nabla f(x(k))]$, enjoys the following bound:
\begin{equation}
    f(x(k))-f(x^*) \leq Q_f \left(1-1/Q_f\right)^k[f(x(0))-f(x^*)]
\end{equation}
where $Q_f=1/L$ and $x^*=\argmin_{\mathbb Q} f(x)$.
\end{proposition}
\begin{proof}
Firstly, it can be verified that the projected gradient descent update is equivalent to
$$x(k)=\argmin_{x\in \mathbb Q}[f(x(k-1)+\nabla f(x(k-1))^\top (x-x(k-1)) + \frac{L}{2}\|x-x({k-1})\|^2]$$
Therefore, we have
\begin{align*}
    f(x(k))&\leq f(x(k-1))+ \nabla f(x(k-1))^\top (x(k)-x(k-1) )+ \frac{L}{2}\|x(k)-x(k-1)\|^2\\
    & =\min_{x\in \mathbb Q}[f(x(k-1)+\nabla f(x(k-1))^\top (x-x(k-1)) + \frac{L}{2}\|x-x(k-1)\|^2]\\
    & \leq f(x(k-1)+\nabla f(x(k-1))^\top (x^*-x(k-1)) + \frac{L}{2}\|x^*-x(k-1)\|^2\\
    & \leq f(x^*) + \frac{L-\mu}{2} \|x^*-x(k-1)\|^2\\
    & = f(x^*) + \frac{L}{2}(1-1/Q_f)\|x^*-x(k-1)\|^2
\end{align*}
where the first inequality is by $L$ smoothness and the last inequality is by $\alpha$ strong convexity.

Consequently, we have
\begin{align*}
    f(x(k))-f(x^*) &\leq \frac{L}{2}(1-1/Q_f)\|x^*-x(k-1)\|^2\\
    & \leq \frac{L}{2}(1-1/Q_f)^k\|x^*-x(0)\|^2\\
    & \leq Q_f (1-1/Q_f)^k(f(x(0))-f(x^*))
\end{align*}
where the second inequality uses the established convergence rate in Theorem 2.3.4 in \cite{nesterov2013introductory}, i.e.  
    $$ \|x(k)-x^*\|^2 \leq(1-1/Q_f)^k\|x(0)-x^*\|^2,$$
    and the last inequality uses the strong convexity. 
\end{proof}

\subsection{Proof of the inequality \eqref{equ: ogd bdd} in Theorem \ref{thm: dynamic regret}}\label{append: ogd proof}

For simplicity, we slightly abuse the notation and denote  OGD's output $\mathbf x(0)$ as $\mathbf x$ below. We first prove a  lemma and then  prove \eqref{equ: ogd bdd}.% The OGD update is thus denoted as $
%  x_t=\Pi_{\mathbb X}[x_{t-1}-\gamma \nabla f_{t-1}(x_{t-1})]$ for $t\geq 2$.% $x_1=x_0$.
% We first prove a  lemma and then  prove \eqref{equ: ogd bdd}. 

% {Try to make it half a page, by adding the comments of the reviewer.}

%Before the proof, we first provide a supportive lemma.
\begin{lemma}\label{switching cost}
	Given stepsize $\gamma=1/l$, the outputs of OGD satisfy $\sum_{t=1}^T\|x_t-\theta_t\|\leq \frac{1}{1-\kappa}\sum_{t=1}^T\|\theta_t-\theta_{t-1}\|$ and\\ $\sum_{t=1}^T\|x_t-x_{t-1}\|^2\leq\frac{2G}{l(1-\kappa)}\sum_{t=1}^T \|\theta_t-\theta_{t-1}\| $,\\
	where  $x_1=\theta_0=x_0$, $\kappa=\sqrt{1-\frac{\alpha}{l}}$, and $\theta_t=\argmin_{\mathbb X}f_t(x)$. % for $t\geq 1$.
\end{lemma}

\nit{Proof. }Firstly, we bound $\sum_{t=1}^T \|x_t-\theta_t\|$ by the following.
{\small\begin{align}
	\sum_{t=1}^T \|x_t-\theta_t\|& \leq \sum_{t=1}^T\|x_t-\theta_{t-1}\|+ \sum_{t=1}^T \|\theta_{t-1}-\theta_{t}\|\nonumber\\
	& = \sum_{t=2}^T\|x_t-\theta_{t-1}\|+ \sum_{t=1}^T \|\theta_{t-1}-\theta_{t}\|\nonumber\\
	& \leq  \sum_{t=2}^T\kappa \|x_{t-1}-\theta_{t-1}\|+ \sum_{t=1}^T \|\theta_{t-1}-\theta_{t}\|%\nonumber\\
	% & \leq \sum_{t=1}^T\kappa \|x_{t}-\theta_{t}\|+ \sum_{t=1}^T \|\theta_{t-1}-\theta_{t}\|
	\label{equ: last line}
	%\label{equ: kappa}
	\end{align}}
where  the equality is by $x_1=\theta_0$, the last inequality is by OGD's updating rule and Theorem 2.2.8  in \cite{nesterov2013introductory}.\footnote{Notice that $g_Q(\cdot, \cdot)$ in \cite{nesterov2013introductory} satisfies $g_Q(x_{t-1}, l)= l(x_{t-1}-x_t)$.} Then, the bound is proved by  subtracting $\sum_{t=1}^T\kappa \|x_{t}-\theta_{t}\|$ from the first term and \eqref{equ: last line}, then dividing both sides by ${1-\kappa}$. % on both sides of the last inequality yield the bound \eqref{equ:3}. 

% {need to explain somewhere that $g_Q$ is our update.}

Next, we bound $\sum_{t=1}^T\|x_t-x_{t-1}\|^2$ by the following.
{\small
	\begin{align*}
	 \frac{l}{2}\sum_{t=1}^T\|x_t-x_{t-1}\|^2&= \frac{l}{2}\sum_{t=2}^T \|x_t-x_{t-1}\|^2=\frac{l}{2}\sum_{t=1}^{T-1} \|x_{t+1}-x_{t}\|^2\\
	& \leq  \sum_{t=1}^{T-1}[ f_{t}(x_{t})-f_{t}(x_{t+1})] \leq  \sum_{t=1}^{T-1}[ f_{t}(x_{t})-f_{t}(\theta_{t}) ]\\
	%&\leq \sum_{t=2}^T \langle \nabla f_{t-1}(x_{t-1}), x_{t-1}-\theta_{t-1}\rangle \\
	& \leq   \sum_{t=1}^{T-1} G \left\| x_{t}-\theta_{t}\right\| \leq \frac{G}{1-\kappa} \sum_{t=1}^T\|\theta_t-\theta_{t-1}\|
	\end{align*}}
where the first line uses $x_1=x_0$, the first inequality uses Corollary 2.2.1 (2.2.16) in \cite{nesterov2013introductory},\footnote{ {This is by $\bar x=x_{t}$, $f=f_{t}$, $x_Q(x_{t}, l)=x_{t+1}$, $g_Q(x_t,l)=l(x_t-x_{t+1})$.}} the second inequality is by $\theta_{t}=\argmin_{\mathbb X} f_{t}(x)$, the third inequality is by $f_{t}(x_{t})-f_{t}(\theta_{t})\leq  \langle \nabla f_{t}(x_{t}), x_{t}-\theta_{t}\rangle\leq G \left\| x_{t}-\theta_{t}\right\| $, the last one uses the first bound in Lemma \ref{switching cost}.\qed

%the fourth line is by $f_{t-1}$'s convexity, the last one is by 
%the gradient bound $G$ in Assumption \ref{ass:1} and the first bound in this lemma. \qed

Now, we can prove \eqref{equ: ogd bdd} in Theorem \ref{thm: dynamic regret} as follows.
\begin{align*}
\text{Reg}(OGD)&\leq \sum_{t=1}^T \left(f_t(x_t)-f_t(x_t^*)+\frac{\beta}{2}\|x_t-x_{t-1}\|^2\right) \\
&\leq \sum_{t=1}^T \left(f_t(x_t)-f_t(\theta_t)+\frac{\beta}{2}\|x_t-x_{t-1}\|^2\right) \\
&\leq \sum_{t=1}^{T}\left[G\|x_t-\theta_t\|+\frac{\beta}{2}\|x_t-x_{t-1}\|^2\right]\\
&\leq \delta\sum_{t=1}^T \|\theta_t-\theta_{t-1}\| %=\delta \sum_{t=1}^T \|\theta_t-\theta_{t-1}\|
\end{align*}
where  the first  line uses $\|x_t^*-x_{t-1}^*\|^2\geq 0$, the second line is by $\theta_t=\argmin_{\mathbb X} f_t(x)$ and thus $f_t(x_t^*)\geq f_t(\theta_t)$, the third inequality is by $f_t(x_t)-f_t(\theta_t)\leq \nabla f_t(x_t)^\top (x_t-\theta_t) \leq G \|x_t-\theta_t\|$, and the last inequality is by Lemma \ref{switching cost}.% and $\delta$'s definition in Theorem \ref{thm: dynamic regret}. 

%$=(\beta/l+1)\frac{G}{1-\kappa}$. %\qed

\subsection{Proof of Corollary \ref{cor: asym stability}}\label{aped: coro proof}
%This section provides a proof of Corollary \ref{cor: asym stability} based on Theorem \ref{thm: dynamic regret}.

The equivalence between \eqref{equ: opt. control} and \eqref{equ: soco} is straightforward. $\theta_t$ is the optimal steady state at $t$ since $$\theta_t=\argmin_{x=x+u\in \mathbb X} [f_t(x)+\beta/2\|u\|^2].$$
We denote $\Theta \coloneqq\sum_{t=1}^{+\infty} \|\theta_t -\theta_{t-1}\| $ and consider RHGD below. For any $T<\infty$, 
\begin{align*}
\sum_{t=1}^{T-W}\frac{\alpha}{2}\|x_t(W)-\theta_t\|^2  & \leq \sum_{t=1}^{T} \left[f_t(x_t(W)) -f_t(\theta_t)\right]\\
& \leq \mathsf{C_T}(\mathbf x(W)) - \mathsf{C_T}(\bm \theta)+ \frac{\beta}{2}\sum_{t=1}^T\|\theta_t-\theta_{t-1}\|^2\\
& \leq \text{Reg}(RHGD)+ \frac{\beta}{2}\sum_{t=1}^T\|\theta_t-\theta_{t-1}\|^2 \\
&\leq \delta Q_f \left[1-\frac{1}{Q_f}\right]^W \Theta+ \frac{\beta}{2}\Theta^2
\end{align*}
where  the first term only considers the summation  to $T-W$ because the outputs $x_t(W)$ for $t>T-W$ depend on whether RHGD terminates at $T$ or not; the first inequality is by the $\alpha$-strong convexity of $f_t(\cdot)$ and the optimality of $\theta_t$; 
the second one is by  \eqref{equ: soco} and $\|x_t(W)-x_{t-1}(W)\|^2\geq 0$, and we denote $\bm \theta=(\theta_1, \dots, \theta_T)$; the third inequality is by $\mathsf{C_T}(\bm \theta)\geq  \mathsf{C_T}(\mathbf x^*)$ and \eqref{equ: dynamic regret def}; the last line uses Theorem \ref{thm: dynamic regret}, $\sum_{t=1}^T\|\theta_t-\theta_{t-1}\|\leq \Theta$, and  $\sum_{t=1}^T\|\theta_t-\theta_{t-1}\|^2\leq (\sum_{t=1}^T\|\theta_t-\theta_{t-1}\|)^2\leq \Theta^2$. 

By letting $T\to +\infty$, we have $\sum_{t=1}^{\infty}\frac{\alpha}{2} \|x_t(W)-\theta_t\|^2 \leq \delta Q_f(1-1/Q_f)^W\Theta+\frac{\beta}{2} \Theta^2 <+\infty$, thus $\|x_t(W)-\theta_t\|\to 0$ as $t\to +\infty$. Since $\theta_t\to \theta_\infty$, we have $\|x_t(W)-\theta_{\infty}\|\to 0$. The proof for RHAG is the same. %\qed

\subsection{Proof of Theorem \ref{thm: W>=1 lower bound}}\label{append: proof of Thm 3 W>=1 LT>=D}
%This section  discusses  $L_T \geq 2D$. The proof for $L_T<2D$ is very similar and deferred to Appendix ??. 
%for technical reasons. We defer the proof for $L_T<2D$ and the proof of the technical lemmas to \cite{li2018online} for space issues.

%\footnote{Due to technical reasons, the proof is slightly different for $L_T < 2D$  and  $L_T\geq 2D$, where the factor 2  can be replaced by any value higher than 1. %, and the only change in the lower bound will be the constant factor. }
%The remaining proof for Theorem \ref{thm: W>=1 lower bound} and the proofs of the technical lemmas in this section are deferred to \cite{li2018online} due to space limitations.

% {(Q: can I revise the complete proof to general $\mathbb X$? Yes, but pay attention to the $0 \in \mathbb X$ case, I really need to assume it, otherwise Lemma 4 no longer holds, has an extra 0 term.)}

% {The only assumption I need to assume that is $0 \in \mathbb X$ and $0=\frac{\upsilon_1+\upsilon_2}{2}$. But this can be achieved by moving the coordinate (this is done by revising $f(x)$'s assumption to move the center.}

% {To do: say for simplicity, assume $n=1$, for general $n>1$, still hold, just with block-diagonal matrix in Lemma 4. We actually consider this case in \cite{li2019online}}

The main proof idea is to construct  random   cost functions and show that the lower bound \eqref{equ: lower bound} holds in expectation, and thus there must exist some case with positive probability such that the lower bound holds.

Without loss of generality, we consider an one dimension case $\mathbb X=[-\frac{D}{2}, \frac{D}{2}]$ and let $x_0=0$.\footnote{{Our proof can be extended to any $\mathbb X\subseteq \R^{n}$  by letting $f_t(x_t)=\alpha/2 \|x_t-\theta_t - \frac{\upsilon_1+\upsilon_2}{2}\|^2$, where $\upsilon_1,\upsilon_2\in \mathbb X$ and $\|\upsilon_1-\upsilon_2\|=D$, and by generating $\theta_t$  randomly   with probability $\mathbb P(\theta=\upsilon_1)=\mathbb P(\theta=\upsilon_2)=1/2$.}} %We let $x_0=0$ for simplicity. %The proofs of the technical lemmas in this section are provided  in \cite{li2018online}.

%	and $\mathbb X=[-\frac{D}{2}, \frac{D}{2}]$ with diameter $D$.\footnote{\blue{Our proof can be extended to $\mathbb X\subseteq \R^{n}$  by letting $f_t(x_t)=\alpha/2 \|x_t-\theta_t\|^2$, where $\theta_t$ is randomly  selected from $\{\upsilon_1,\upsilon_2\} \subseteq \mathbb X$  s.t. $\|\upsilon_1-\upsilon_2\|=D$.}} The proofs of lemmas below are in \cite{li2018online}.

For technical reasons, we consider three cases: (i) when $L_T\geq 2D$, (ii) when $D<L_T <2D$, (iii) when $0\leq L_T \leq D$, and construct slightly different cost function sequences for each case.

\noindent\underline{Proof for (i) when $L_T\geq 2D$:}

\nbf{Part 1: construct random $\{f_t(\cdot)\}_{t=1}^T$: }
For any  $\alpha >0$, $\beta>0$ ($\beta=0$ is trivial), we construct quadratic function
$
f_t(x_t)=\frac{\alpha}{2}(x_t-\theta_t)^2
$
with parameter $\theta_t \in \mathbb X$. Thus, $\theta_t = \argmin_{\mathbb X} f_t(x_t)$.
We construct random $\{\theta_t\}_{t=1}^T$ as follows. Let $\Delta=\ceil{T/\floor{L_T/D}}$ and divide  $T$ stages into $K=\ceil{\frac{T}{\Delta}}$ parts,
where each part has $\Delta$ stages,  except possibly the last part. Notice that $1\leq K \leq T$ since $1 \leq \Delta\leq T$ when $D\leq L_T\leq DT$. For $0\leq k \leq K-1$,  generate $\theta_{k\Delta+1}$ independently by distribution $\mathbb P(\theta=\frac{D}{2})=\mathbb P(\theta=-\frac{D}{2})=\frac{1}{2}$, and let $\theta_t=\theta_{k\Delta+1}$ for $k\Delta+2\leq t \leq (k+1)\Delta$.%, which satisfies the next lemma.   %We verify the path length bound below. %It can be verified that the path length of the constructed sequence  is no more than $L_T$.

The next lemma shows that our constructed cost functions satisfy the path length upper bound. 
\begin{lemma}\label{lem: path length <LT}
	For  $f_t(x_t)$ constructed above, the path length of $\{f_t\}_{t=1}^T$ satisfies $\sum_{t=1}^T \|\theta_t-\theta_{t-1}\|\leq L_T$, where $\theta_0=x_0=0$. 
\end{lemma}
%	The proof is in \cite{li2018online}.
\begin{proof}

	 By the definition of $\Delta$, we have
	\begin{align*}
	\Delta&=\ceil{T/\floor{L_T/D}} \geq T/\floor{L_T/D} 
	\end{align*}
	Equivalently, we have
	$ \floor{L_T/D}  \geq T/\Delta$ and thus $ \floor{L_T/D}  \geq \ceil{T/\Delta}=K$. 
	
		Then, according to the definition of $\theta_t$ above, we have
	\begin{align*}
	\sum_{t=1}^T \| \theta_t- \theta_{t-1}\|  & = \sum_{k=0}^{K-1} \| \theta_{k\Delta+1}- \theta_{k\Delta}\|  \leq DK  \leq D\floor{L_T/D}=L_T
	\end{align*}
	which completes the proof.

	%Since $K = \ceil{\frac{T}{\Delta}}= \min\{i\in \Z |\ i \geq T/\Delta \}$, and $\floor{L_T/D}\in \Z$, we have 
	%$ K \leq \floor{L_T/D}\leq L_T/D$
\end{proof}

\nbf{Part 2: characterize $\mathbf x^*$.} The  problem  \eqref{equ: soco} constructed in Part 1 enjoys a closed-form optimal solution: $\mathbf x^* =  \bm \theta \mathbf A^{\top}$, where 
$\mathbf x^*=(x^*_1, \dots, x^*_T)$ denotes the optimal solution as a row vector in $\R^T$ $\bm \theta=(\theta_1, \dots, \theta_T)$, $\mathbf A=(a_{i,j})_{i,j=1}^T$. Equivalently, we have $x_t^*=\sum_{\tau=1-t}^{T-t} a_{t, t+\tau} \theta_{t+\tau}$ so $a_{t, t+\tau}$  represents the influence of $\theta_{t,t+\tau}$ on $x_t^*$. Lemma \ref{lem: x*=Atheta} will show that the influence  $a_{t, t+\tau}$ decays exponentially for $\tau\geq 0$.% the influences of the future parameters decays
\begin{lemma}\label{lem: x*=Atheta}
	Consider the cost in Part 1. The optimal solution to \eqref{equ: soco} is $\mathbf x^* =  \bm \theta \mathbf A^{\top}$, where 
	$a_{t,t+\tau} \geq \frac{\alpha}{\alpha+\beta}(1-\rho)\rho^{\tau}$ for $\tau\geq 0$.%,
	%	where $\rho = \frac{\sqrt{Q_f}-1}{\sqrt{Q_f}+1}$.
\end{lemma}
The proof is provided in Appendix \ref{apen: x*=Atheta}.

\nbf{Part 3: characterize  $\mathbf x^{\A}$.}
The key observation here is that the output $x_t^{\A}$ of any online algorithm $\A$
is a random variable determined by $\{\theta_s\}_{s=1}^{t+W-1}$. This is because $x_t^{\A}$ is decided by $\A$ based on
$\{ f_s(\cdot)\}_{s=1}^{t+W-1}$, which is determined by $\{\theta_s\}_{s=1}^{t+W-1}$ according to the construction in Step 1.\footnote{ {Rigorously speaking, to ensure the random variable to be well-defined, some measurability assumption on $\A$ should be imposed, which is satisfied by most algorithms in practice and is thus omitted here for simplicity.}} %which is omitted here for simplicity. We note that most algorithms  to ensure that the random variable is well-defined which is satisfied by most algorithms in practice.}}

\nbf{Part 4: lower bound }$\E [\text{Reg}(\A)]$\textbf{.} To prove the lower bound, we define a set of stages
onsider  a set of stages 
$$\mathbb  J  \coloneqq \{1\leq t\leq T-W\mid  t+W \equiv 1\pmod \Delta \}.$$
Before the proof of the lower bound, we first prove two helping lemmas.
\begin{lemma}\label{lem: E(xtA-xt*)^2}
	Consider the cost in  Part 1, 
	for any online  algorithm $\A$, we have
	$\E | x_t^{\A} -x_t^*|^2 \geq \frac{a_{t,t+W}^2 D^2}{4}$ for any $t\in \mathbb J$, 
	where $a_{t,t+W}$ is an entry of matrix $\mathbf A$ defined in Lemma \ref{lem: x*=Atheta}.
	
\end{lemma}

\begin{proof}
	We denote the $\sigma$-algebra generated by $\theta_1, \dots, \theta_{t}$ as $\F_t$. By our discussion in Part 2, $x_t^*$ is $\F_T$-measurable. In addition, by Part 3, for any online algorithm $\A$, the output $x_t^{\A}$ is $\F_{t+W-1}$-measurable. It is a classic result that for any $\sigma$-algebra $\F$ of the probability space,  the conditional expectation $\E[X\mid \F]$ minimizes the mean square error $E(Y-X)^2$ among any random variable $Y$ that is  $\F$-measurable (see Theorem 4.1.15 in \cite{durrett2019probability} for example). Therefore, we have
	$$ \E ( x_t^{\A} -x_t^*)^2 \geq  \E\left( \E[x_t^*\mid \F_{t+W-1}]-x_t^*\right)^2$$
	Consequently, we do not have to discuss each online algorithm $\A$ but only need to bound  $\E\left( \E[x_t^*\mid \F_{t+W-1}]-x_t^*\right)^2$, which is provided below.
		\begin{align*}
	 \E\left( \E[x_t^*\mid \F_{t+W-1}]-x_t^*\right)^2&= \E\left( \sum_{\tau=1}^{t+W-1}a_{t,\tau} \theta_{\tau}-\sum_{\tau=1}^{T}a_{t,\tau} \theta_{\tau}\right)^2 =\E\left(\sum_{\tau=t+W}^{T}a_{t,\tau} \theta_{\tau}\right)^2\\
	& = \E[(a_{t,t+W}+\dots +a_{t,t+W+\Delta-1})\theta_{t+W}]^2 + \E\left(\sum_{\tau=t+W+\Delta}^T \theta_{\tau}\right)^2\\
	& \geq \E [a_{t,t+W}^2 \theta_{t+W}^2] = a_{t,t+W}^2 \frac{D^2}{4}
	\end{align*}
	where the first line  is because $x_t^*=\sum_{\tau=1}^{T}a_{t,\tau}\theta_{\tau}$ as discussed in Part 2 and $\E[\theta_{\tau}\mid \F_{t+W-1}]=0$ for $\tau\geq t+W$ when  $t\in \mathbb J$ according to our definition of $\{\theta_t\}_{t=1}^T$ in Part 1; the second line is because when $t\in \mathbb J$, we have $\theta_{t+W}=\dots =\theta_{t+W+\Delta-1}$ with mean 0 and $\theta_{t+W}$ is independent of $\theta_{\tau}$ for $\tau\geq t+W+\Delta$ by Part 1; the third line is because $a_{t,\tau}>0$  by Lemma \ref{lem: x*=Atheta}, and $\E[\theta_{t+W}^2]=\frac{D^2}{4}$ by Part 1. 
\end{proof}

\begin{lemma} \label{lem: W>0, LT>2D |J|}
	When $T \geq 2W$ and $ L_T \geq 2D$, we have 
	$ |\mathbb J| \geq \frac{L_T}{12D}$.
\end{lemma}
\begin{proof}
	By using the definition of $\mathbb J$ and the properties of the floor and ceiling operators, we have
		\begin{align*}
		|\mathbb J| &= \ceil{T/\Delta}- \ceil{W/\Delta} \geq \floor{\frac{T-W}{\Delta}}  \geq \frac{T-W}{2\Delta}\\
	& \geq \frac{1}{2} \frac{T-W}{T/\floor{L_T/D}+1} = \frac{1}{2}\floor{L_T/D}\frac{T-W}{T+\floor{L_T/D}}\\
	& \geq \frac{1}{2} \floor{L_T/D}  \frac{T-T/2}{T+T} = \frac{1}{8} \floor{L_T/D} \geq \frac{1}{12}L_T/D
	\end{align*}
	where the first equality can be proved by noticing that 
	\begin{align*}
	|\mathbb J|&=|\{W+1\leq\tau\leq T\mid \tau \equiv 1 (\text{mod}) \Delta \}|\\
	&=|\{1\leq \tau \leq T \mid  \tau \equiv 1(\text{mod})  \Delta\}| - \{1\leq \tau \leq T \mid  \tau \equiv 1(\text{mod})  \Delta\}|  \\
&=\ceil{T/\Delta}- \ceil{W/\Delta} ;
	\end{align*}
	the first inequality is a property of floor and ceiling functons; the second inequality uses the fact that $\floor{x}\geq x/2$ when $x\geq 1$, and that $T-W\geq \Delta$ when $L_T\geq 2D$ and $T\geq 2W$ because we have $\frac{T}{T-W}\leq 2 \leq \floor{L_T/D}$ and thus $T-W=\ceil{T-W}\geq \ceil{T/\floor{L_T/D}}=\Delta$; the third inequality is by $\Delta = \ceil{T/\floor{L_T/D}}\leq T/\floor{L_T/D} +1$, the fourth inequality is by $T \geq 2W$ and $L_T \leq DT$, and the last inequality is because $L_T/ D \geq 2$, and $\floor{x}\geq \frac{2}{3} x$ when $x\geq 2$.
\end{proof}
%	 $T\geq 2W$, we have $\frac{T}{T-W}\leq 2$, when 
%	
%	
%	
%	; the third inequality is by $\Delta = \ceil{T/\floor{L_T/D}}\leq T/\floor{L_T/D} +1$, the fourth inequality is by $T \geq 2W$ and $L_T \leq DT$, and the last inequality is because $L_T/ D \geq 2$, and $\floor{x}\geq \frac{2}{3} x$ when $x\geq 2$.
%
%	
%	 after rewriting the set $J$, the first inequality is a property of floor and ceiling functons, the second inequality is by Lemma \ref{lem: W>0 LT>2D bound Delta}, the third inequality is by $\Delta = \ceil{T/\floor{L_T/D}}\leq T/\floor{L_T/D} +1$, the fourth inequality is by $T \geq 2W$ and $L_T \leq DT$, and the last inequality is because $L_T/ D \geq 2$, and $\floor{x}\geq \frac{2}{3} x$ when $x\geq 2$.
%	

Based on the helping lemma above, we can lower bound the regret  in expectation. 
\begin{align*}
\E [\text{Reg}(\A)]&=\E[\mathsf{C_T}(\mathbf x^{\A})-\mathsf{C_T}(\mathbf x^*)] \geq \frac{\alpha}{2}\E\|\mathbf x^{\A}-\mathbf x^*\|^2\\
& = \frac{\alpha}{2} \sum_{t=1}^T \E |x_t^{\A}-x_t^*|^2   \geq \frac{\alpha}{2}\sum_{t\in \mathbb J} \E |x_t^{\A}-x_t^*|^2 \\
& \geq |\mathbb J|\frac{a_{t,t+W}^2 D^2\alpha}{8}
\geq  \frac{\alpha D }{96} (1-\rho)^2 \left(\frac{\alpha}{\alpha+ \beta}\right)^2L_T\rho^{2W}
\end{align*}
where the first inequality uses  Lemma \ref{lem:cost}, the third one uses Lemma \ref{lem: E(xtA-xt*)^2},  the last one uses Lemma  \ref{lem: W>0, LT>2D |J|} and \ref{lem: x*=Atheta}. Thus, there  exists some realization of $\bm \theta$ yielding the lower bound on the regret, which completes the proof when $L_T\geq 2D$.  % that the lower bound hold.

\noindent\underline{Proof for (ii) when $D< L_T<2D$:} The proof is very similar to the proof above. We also consider cost function $f_t(x_t)=\frac{\alpha}{2}(x_t-\theta_t)^2$, but we define $\{\theta_t\}_{t=1}^T$ in a slightly different way, i.e. we let $\theta_1=\dots=\theta_W=0$, and $\theta_{W+1}=\dots=\theta_T$ be a random variable following distribution $\Pb(\theta_t=\frac{D}{2})=\Pb(\theta_t= -\frac{D}{2})=\frac{1}{2}$. 	It is easy to verify that the upper bound $L_T$ on the path length is satisfied: $\sum_{t=1}^T\| \theta_t - \theta_{t-1}\| = \| \theta_{W+1}\| = \frac{D}{2}\leq \frac{L_T }{2}\leq L_T$. Since the matrix $\mathbf A$ does not depend on our choices of $\{\theta_t\}_{t=1}^T$, Lemma \ref{lem: x*=Atheta} still holds. In addition, similar to Lemma \ref{lem: E(xtA-xt*)^2}, we have $ \E |x_1^\A-x_1^*|^2 \geq \frac{a_{1,1+W}^2 D^2}{4}$.

Consequently, we have the lower bound for the regret in expectation by
\begin{align*}
\E [\text{Reg}(\A)]&=\E[\mathsf{C_T}(\mathbf x^{\A})-\mathsf{C_T}(\mathbf x^*)] \geq \frac{\alpha}{2}\E|\mathbf x^{\A}-\mathbf x^*|^2 \geq \frac{\alpha}{2} \E |x_1^{\A}-x_1^*|^2  \\
&\geq \frac{\alpha a_{1,1+W}^2D^2}{8}  \geq \frac{\alpha D^2}{8} \rho^{2W}(1-\rho)^2 \left(\frac{\alpha}{\alpha+ \beta}\right)^2 \geq \frac{\alpha DL_T}{96} \rho^{2W}\left(\frac{\alpha(1-\rho)}{\alpha+ \beta}\right)^2
\end{align*}

\noindent\underline{Proof for (iii) when $0\leq L_T\leq D$:} The proof is very similar to the proof above. We also consider cost function $f_t(x_t)=\frac{\alpha}{2}(x_t-\theta_t)^2$, but we define $\{\theta_t\}_{t=1}^T$ in a slightly different way, i.e. we let $\theta_1=\dots=\theta_W=0$, and $\theta_{W+1}=\dots=\theta_T$ be a random variable following distribution $\Pb(\theta_t=\frac{L_T}{2})=\Pb(\theta_t= -\frac{L_T}{2})=\frac{1}{2}$. 	It is easy to verify that the upper bound $L_T$ on the path length is satisfied: $\sum_{t=1}^T\| \theta_t - \theta_{t-1}\| = \| \theta_{W+1}\| = \frac{L_T }{2}\leq L_T$. Since the matrix $\mathbf A$ does not depend on our choices of $\{\theta_t\}_{t=1}^T$, Lemma \ref{lem: x*=Atheta} still holds. In addition, similar to Lemma \ref{lem: E(xtA-xt*)^2}, we have $ \E |x_1^\A-x_1^*|^2 \geq \frac{a_{1,1+W}^2 L_T^2}{4}$.

Consequently, we have the lower bound for the regret in expectation by
\begin{align*}
\E [\text{Reg}(\A)]&=\E[\mathsf{C_T}(\mathbf x^{\A})-\mathsf{C_T}(\mathbf x^*)] \geq \frac{\alpha}{2}\E|\mathbf x^{\A}-\mathbf x^*|^2 \geq \frac{\alpha}{2} \E |x_1^{\A}-x_1^*|^2  \\
&\geq \frac{\alpha a_{1,1+W}^2L_T^2}{8}  \geq \frac{\alpha L_T^2}{8} \rho^{2W}(1-\rho)^2 \left(\frac{\alpha}{\alpha+ \beta}\right)^2 \geq \frac{\alpha L^2_T}{96} \rho^{2W}\left(\frac{\alpha(1-\rho)}{\alpha+ \beta}\right)^2
\end{align*}

%
%
%where the last inequality is not tight
%
%
%By Lemma \ref{lem: jump once E(xtA-xt*)^2} to be stated below, we have $\E \|x^\A-x^*\|^2 \geq\E \|x_1^\A-x_1^*\|^2 \geq \frac{a_{1,1+W}^2 D^2}{4}$. 
%As a result, there must exist a sequence such that 
%\begin{align*}
%&\mathsf{C_T}(x^{\A})-\mathsf{C_T}(x^*) \geq \frac{\alpha}{2} \|x^{\A}-x^*\|^2 \\
%& \geq \frac{\alpha D^2}{8} \rho^{2W}(1-\rho)^2 \left(\frac{\alpha}{\alpha+ \beta}\right)^2 \geq \frac{\alpha DL_T}{96} \rho^{2W}\left(\frac{\alpha(1-\rho)}{\alpha+ \beta}\right)^2
%\end{align*}
%The proof is done.
%

%
%\subsection{Proof of Lemma \ref{lem: path length <LT}:} \label{apen:  path length <LT}
%
%%	\nit{Proof of Lemma \ref{lem: path length <LT}:}
%	
%	
%	
%	According to the construction, 	for any realization of $\theta_t$, we have
%	\begin{align*}
%	\sum_{t=1}^T \| \theta_t- \theta_{t-1}\|  & = \sum_{k=0}^{K-1} \| \theta_{k\Delta+1}- \theta_{k\Delta}\|  \leq DK 
%	\end{align*}
%	
%	
%	In the following, we will show that $K\leq L_T/D$, then the proof is done. Remember the definition of $\Delta$: 
%	\begin{align*}
%	\Delta&=\ceil{T/\floor{L_T/D}} \geq T/\floor{L_T/D} 
%	\end{align*}
%	Equivalently, 
%	$ \floor{L_T/D}  \geq T/\Delta$
%	Since $K = \ceil{\frac{T}{\Delta}}= \min\{i\in \Z |\ i \geq T/\Delta \}$, and $\floor{L_T/D}\in \Z$, we have 
%	$ K \leq \floor{L_T/D}\leq L_T/D$
%	\qed

\subsection{Proof of Lemma \ref{lem: x*=Atheta}}\label{apen: x*=Atheta}
The proof takes four steps:
\begin{enumerate} [(I)]
	\item study unconstrained optimization and show that $\bm{\tilde x}^* = \argmin_{
		\R^T} \mathsf{C_T}(\bm x)=\bm \theta\bm A^\top$.
	\item show that the constrained optimization admits the same optimal solution: $\bm x^*= \bm{\tilde x}^*$
	\item give closed-form expression for matrix $\bm A$
	\item lower bound the entries $a_{t,t+\tau}$ for $\tau\geq 0$ of matrix $\bm A$
\end{enumerate}

\nbf{(I) Unconstrained optimization $\argmin_{
		\R^T} \mathsf{C_T}(x)=A\theta$.}\\
By \eqref{equ: soco} and Part 1 in Appendix \ref{append: proof of Thm 3 W>=1 LT>=D}, we have that
\begin{align*}
\mathsf{C_T} (\bm x) &= \sum_{t=1}^T \left[  f_t(x_t) + \frac{\beta}{2} \| x_t - x_{t-1}\|^2 \right] = \sum_{t=1}^T \left[ \frac{\alpha}{2} \|x_t - \theta_t \|^2 + \frac{\beta}{2} \| x_t - x_{t-1}\|^2 \right]. 
\end{align*} 
Since $\mathsf{C_T}(\bm x)$ is strongly convex, the optimal solution to  $\min_{\R^T} \mathsf{C_T}(\bm x)$ can be defined by the first-order condition below:
\begin{align*}
& \alpha (\tilde x_t- \theta_t) + \beta(2 \tilde x_t-\tilde x_{t-1}-\tilde x_{t+1})=0,\quad t=1,\dots, T-1\\
&\alpha (\tilde x_T- \theta_T) + \beta( \tilde x_T-\tilde x_{T-1})=0
\end{align*}
By $x_0 = \theta_0=0$ and canceling  $\alpha$ on both sides, we can write the linear equation systems in the  matrix form:
$  \bm{\tilde  x}^*\bm H^\top  = \bm \theta$,
where $\bm x, \bm\theta$ are row vectors, $\bm H$ is given as below:
%\begin{align*}
%&	h_{tt} = 1+2 \frac{\beta}{\alpha}, \quad 1\leq t \leq T-1\\
%&   h_{TT}= 1+ \frac{\beta}{\alpha}\\
%&   h_{t,t-1} = h_{t, t+1} = -\frac{\beta}{\alpha} \\
%&   h_{t, s} =0, \quad \text{otherwise}
%\end{align*}
\begin{equation}\label{equ:: H}
\bm H=
\begin{pmatrix} 
1+2 \frac{\beta}{\alpha}& -\frac{\beta}{\alpha} &0  & \cdots   &0 \\
-\frac{\beta}{\alpha} & 1+2\frac{\beta}{\alpha} & -\frac{\beta}{\alpha} &\cdots  &0\\
0 &	-\frac{\beta}{\alpha} & 1+2\frac{\beta}{\alpha}  &\cdots  &0\\
\vdots &\vdots &  \vdots &\ddots   & \vdots  \\
0 & 0 & 0  &\cdots & 1+ \frac{\beta}{\alpha}\\
\end{pmatrix}
\end{equation}
Notice that $\bm H$ is strictly diagonally dominant, so $\bm H$ is invertible. Therefore, the optimal solution to the unconstrained optimization, $\bm{\tilde x}^* = \argmin_{
	\R^T} \mathsf{C_T}(\bm x)$, is given by
\begin{align*}
&\bm{\tilde x}^* =\bm \theta \bm A^\top \ \ \ 
\text{where }\bm A\coloneqq \bm H^{-1}
\end{align*}

\nbf{(II) The constrained optimization has the same solution.}
Since $\bm H$ is strictly diagonally dominant, then by Theorem 1 in \cite{varah1975lower}, we have
$$  \|\bm A\|_{\infty}= \|\bm H^{-1}\|_{\infty} \leq  \max_{1\leq t \leq T} \frac{1}{|h_{tt}|-\sum_{s\not = t}|h_{t,s}|}=1$$
Besides, since $\bm H$ has negative off-diagonal entries and positive diagonal entries, and is strictly diagonally dominant, the inverse of $\bm H$, denoted by $\bm A$ now, is nonnegative. 
Therefore, for each $t$, $\tilde x^*_t$ can be written as a convex combination of elements in $\mathbb X$:
$$ \tilde x^*_t = \sum_{s=1}^T a_{t,s}\theta_s + (1- \sum_{s=1}^T a_{t,s})0 $$
because $\theta_t \in \mathbb X $ and $0 \in \mathbb X$. By the convexity of $\mathbb X$, we have $\tilde x^*_t \in \mathbb X$, then naturally, $\tilde x^* \in \mathbb X\times \dots \times \mathbb X$. As a result, $\bm x^* = \argmin_{\mathbb X\times \dots \times \mathbb X} \mathsf{C_T}(\bm x)= \argmin_{\R^T} \mathsf{C_T}(\bm x)=\bm{\tilde x}^*=\bm \theta\bm A^\top $.

\nbf{(III) Closed form expression of $\bm A$.}\\
Since matrix $\bm H$ has many good properties, such as strictly diagonal dominance, positive diagonally entries and negative off-diagonal entries, tridiagonality, symmetry, we can find a closed-form expression for its inverse, denoted by $\bm A$ now, according to Theorem 2 in \cite{concus1985block}. 
In particular, the entries of $A$ are given by $a_{t,t+\tau}= \frac{\alpha }{\beta} w_t v_{t+\tau}$ 	for $\tau\geq 0$ where
	\begin{align*}
	w_t &= \frac{\rho}{1- \rho^2}\left( \frac{1}{\rho^t}   -\rho^t \right)\quad v_T = \frac{1}{-w_{T-1}+ (\xi -1) w_t}\\
	v_t &= c_3 \frac{1}{\rho^{T-t}} + c_4 \rho^{T-t} \qquad	c_3 = v_T\left( \frac{(\xi -1 )\rho - \rho^2}{1-\rho^2}  \right)\qquad c_4  = v_T \frac{1-(\xi-1) \rho}{1-\rho^2}
	\end{align*}
and $\rho = \frac{\sqrt{Q_f}-1}{\sqrt{Q_f}+1}$, $\xi = \alpha/\beta +2$. Since $A$ is nonnegative and $w_t$ is apparently positive, we have $v_t \geq 0$ for all $t$.

\nbf{(IV) Lower bound $a_{t,t+\tau}$ for $\tau\geq 0$.}
\\
We will bound $w_t$, $v_T$ and $v_{t+\tau}/v_T$ separately and then combine them together for a lower bound of $a_{t,t+\tau}$ for $\tau\geq 0$.

	First, we bound $w_t$ by
	$$\rho^t w_t = \frac{\rho}{1-\rho^2}(1-\rho^{2t})\geq \rho$$
	since $t\geq 1$ and $\rho<1$.
	
	Next, we bound $v_T$ in the following way:
	\begin{align*}
	\rho^{-T} v_T & = \frac{1}{(\xi-1)(1-\rho^{2T})-(\rho-\rho^{2T-1}) } \frac{1-\rho^2}{\rho}\\
	& \geq \frac{1}{(\xi-1)(1-\rho^{2T}) } \frac{1-\rho^2}{\rho}  \geq \frac{1}{\xi-1 } \frac{1-\rho^2}{\rho}  
	\end{align*}
	where $\xi = \frac{\alpha }{\beta}+2= \frac{2Q_f +2}{Q_f-1}$, $\rho = \frac{\sqrt{Q_f}-1}{\sqrt{Q_f}+1}$; the first inequality is by $T\geq 1$, $(\rho-\rho^{2T-1})\geq 0$; the second inequality is by $0< \rho<1$.
	
	Then, we bound $v_{t+\tau}/v_T$.
	\begin{align*}
	\rho^{T-t-\tau}\frac{v_{t+\tau}}{v_T}&=  \left( \frac{(\xi -1 )\rho - \rho^2}{1-\rho^2}  \right) + \frac{1-(\xi-1) \rho}{1-\rho^2} \rho^{2(T-t-\tau)}\\
	& \geq \left( \frac{(\xi -1 )\rho - \rho^2}{1-\rho^2}  \right) 
	 = \left( \frac{\rho^2+1-\rho - \rho^2}{1-\rho^2}  \right)  =  \frac{1-\rho }{1-\rho^2}  
	\end{align*}
	where the  inequality is by $1-(\xi-1) \rho \geq 0$, $v_T \geq 0$, and the second equality is by  $\rho^2-\xi\rho+1=0$.
	
	Finally, combining three parts together,
	\begin{align*}
	a_{t,t+\tau}&= \frac{\alpha }{\beta}  \left[\rho^t w_t\right]\left[\rho^{-T}v_T\right] \left[\rho^{T-t-\tau}\frac{v_{t+\tau}}{v_T}\right]\rho^{\tau}\\
&	\geq  \frac{\alpha}{\beta} \rho \frac{1}{\xi-1 } \frac{1-\rho^2}{\rho} \left( \frac{1-\rho }{1-\rho^2}  \right)\rho^{\tau} = \frac{\alpha}{\alpha+\beta}(1-\rho)\rho^{\tau}
	\end{align*}

	\subsection{Proof of Theorem \ref{thm: W=0 lower bound}}\label{apen: W=0 lower bound}
	The proof is similar to the proof of Theorem \ref{thm: W>=1 lower bound}. We will also construct random cost functions and prove the lower bound of the regret in expectation. We will discuss two scenarios: $0 < L_T < D$, and $D \leq L_T \leq DT$ ($L_T=0$ is trivially true), and construct different function sequences to prove the lower bound. Without loss of generality, we let $x_0=0$.
	%	Remember that $0\leq L_T \leq DT$, so we will discuss two scenarios: $0 < L_T < D$, and $D \leq L_T \leq DT$ ($L_T=0$ is trivially true), and construct different function sequences to prove the lower bound. The proof will be very similar to the proof of Theorem \ref{thm: W>=1 lower bound}, we will first construct random sequence, then show that the lower bound holds in expectation. Without loss of generality, we let $x_0=0$.
		
		\nit{Scenario 1: $0 < L_T < D$.} 
		
		\nbf{Construction of random costs.} For each $0 < L_T < D$, we consider the following construction of $\mathbb X \subseteq \R^2$:
		$$\mathbb X= [-\frac{L_T}{2}, \frac{L_T}{2}]\times [ - \frac{\sqrt{D^2-L_T^2}}{2},  \frac{\sqrt{D^2-L_T^2}}{2}]$$
		It is easy to verify that the diameter of $\mathbb X$ is $D$.
		
		For any $\alpha >0$, consider the parametrized cost function:
		$$ f_t (x_t, y_t; \tilde x_t, \tilde y_t) = \frac{\alpha}{2} (x_t-\tilde x_t)^2 + \frac{\alpha}{2} (y_t-\tilde y_t)^2 $$
		where $\tilde x_t \in [-M, M]$  for  $M =D+ (1+ \beta/\alpha)\frac{L_T}{2}$ and $\tilde y_t \in [-\frac{D}{2}, \frac{D}{2}]$. It is easy to verify that the gradient bound is $G= \alpha  \sqrt{(M+D/2)^2 +D^2} \leq (3\alpha +\beta)D$.
		%where $(\tilde x_t, \tilde y_t)\in \R^2$ are parameters which may be outside the action space $\mathbb X$. It is easy to verify that $f_t (x_t, y_t; \tilde x_t, \tilde y_t) $ belongs to function class $\F_X(\alpha, \alpha, G)$, where $G= \alpha  \sqrt{(M+D/2)^2 +D^2}$ when 
	%	$\tilde y_t \in [-\frac{D}{2}, \frac{D}{2}]$ and $\tilde x_t \in [-M, M]$ and $M =D+ (1+ \beta/\alpha)\frac{L_T}{2}$.
		
		Next, we consider two possible function sequences, and each sequence happens with probability 1/2.
		
		Sequence 1: $\tilde x_1 = M$, $\tilde x_t = \frac{L_T}{2}$ for $t\geq 2$. $\tilde y_t =0$, $t \in[T]$.
		
		Sequence 2: $\tilde x_1 = -M$, $\tilde x_t = -\frac{L_T}{2}$ for $t\geq 2$. $\tilde y_t =0$, $t \in[T]$.
		%where $M = D+(1+ \beta/\alpha)L_T/2$.

		Let $(\theta_t, \varphi_t) = \argmin_{\mathbb X} f_t(x_t, y_t; \tilde x_t, \tilde y_t)$, and $(\bm x^*,\bm y^*) \coloneqq (x_1, y_1, \dots, x_T, y_T)= \argmin_{\mathbb X \times \dots \mathbb X} \mathsf{C_T}(\bm x,\bm y)$. Then, for each sequence of the cost functions, the optimal solutions $(\bm x^*,\bm y^*)$ are
		
		Sequence 1: $ \theta_t = x_t^*= \frac{L_T}{2}$, $\varphi_t =  y^*_t =0$ for $1\leq t \leq T$.
		
		Sequence 2: $ \theta_t = x_t^*= -\frac{L_T}{2}$, $\varphi_t =  y^*_t =0$ for $1\leq t \leq T$.

\vspace{6pt}
		\nbf{Bound $\E[\mathsf{C_T}(\bm x^\A,\bm y^\A)-\mathsf{C_T}(\bm x^*,\bm y^*)]$. } \\
		By strong convexity, we have
		\begin{align*}
		\E [\mathsf{C_T}(\bm x^\A,\bm y^\A)-\mathsf{C_T}(\bm x^*,\bm y^*)]&\geq \E \sum_{t=1}^T \left[ \frac{\partial \mathsf{C_T}}{\partial x_t}(\bm x^*,\bm y^*)(x_t^{\A}-x_t^* )+ \frac{\partial \mathsf{C_T}}{\partial y_t}(x^*,y^*)(y_t^{\A}-y_t^* )\right]\\
		& \geq \E  \left[ \frac{\partial \mathsf{C_T}}{\partial x_1}(x^*,y^*)(x_1^{\A}-x_1^* )\right] \\
		& = \frac{1}{2} (-h)(x_1^{\A}- \frac{L_T}{2})+ \frac{1}{2} h(x_1^{\A}+\frac{L_T}{2})  = \frac{1}{2} hL_T\\
			&\geq \frac{\alpha D}{2}   L_T \geq  \frac{\alpha DL_T}{32} (1-\rho)^2 \left(\frac{\alpha}{\alpha+ \beta}\right)^2
		\end{align*}
		where the second inequality is by $\frac{\partial \mathsf{C_T}}{\partial y_t}(\bm x^*,\bm y^*)=0$ when $t \geq 1$, and $\frac{\partial \mathsf{C_T}}{\partial x_t}(\bm x^*,\bm y^*)(x_t^{\A}-x_t^* )=0$ when $t \geq 2$; in the first equality, $h= \frac{\partial \mathsf{C_T}}{\partial x_1}(\bm x^*,\bm y^*) $ when the costs follow Sequence 2, so $\frac{\partial \mathsf{C_T}}{\partial x_1}(\bm x^*,\bm y^*) =-h$ when the costs follow Sequence 1; the third inequality is by $h \geq \alpha D$; the last inequality is to be consistent with the bound in Scenario 2.
		
%		Next, we will expand the expectation by considering two possible sequences:
%		\begin{align*}
%		&\E [\mathsf{C_T}(x^\A,y^\A)-\mathsf{C_T}(x^*,y^*)] \geq \E  \left[ \frac{\partial \mathsf{C_T}}{\partial x_1}(x^*,y^*)(x_1^{\A}-x_1^* )\right]\\
%		& = \frac{1}{2} (-h)(x_1^{\A}- \frac{L_T}{2})+ \frac{1}{2} h(x_1^{\A}+\frac{L_T}{2})  \geq \frac{1}{2} hL_T
%		\end{align*}
%		where $h= \frac{\partial \mathsf{C_T}}{\partial x_1}(x^*,y^*) $ when the costs follow Sequence 2 and $h= -\frac{\partial \mathsf{C_T}}{\partial x_1}(x^*,y^*) $ when the costs follow Sequence 1. When $M =D+ (1+ \beta/\alpha)\frac{L_T}{2}$, $h= \alpha D$ and
%		$$ \E [\mathsf{C_T}(x^\A,y^\A)-\mathsf{C_T}(x^*,y^*)] \geq  \frac{\alpha D}{2}   L_T$$
%		
%		Since  $G= \alpha  \sqrt{(M+D/2)^2 +D^2} \leq \alpha  D \sqrt{(2+ \beta/(2\alpha))^2 +1}$, it is easy to verify that our current lower bound  by showing that
%		\begin{align*}
%	&\frac{\alpha D}{2}   L_T \geq 	\frac{G L_T}{2 \sqrt{(2+ \beta/(2\alpha))^2 +1}}\geq \frac{GL_T}{4 (2+ \beta/(2\alpha)) }\\
%	& \geq \frac{\alpha GL_T}{8(\alpha+\beta) } \geq  \frac{GL_T}{32} (1-\rho)^2 \left(\frac{\alpha}{\alpha+ \beta}\right)^2
%		\end{align*}
	%	$$ \E [\mathsf{C_T}(x^A,y^A)-\mathsf{C_T}(x^*,y^*)] \geq  \frac{G}{2 \sqrt{(2+ \beta/(2\alpha))^2 +1}}   L_T$$
		
%		Therefore, for any $\A$, there must exist a sequence, either Sequence 1 or 2 or both, that makes 
%		$$ \mathsf{C_T}(x^A,y^A)-\mathsf{C_T}(x^*,y^*)\geq  \frac{\alpha^3D   ( 1 - \rho )^2 }{32(\alpha+\beta)^2  }L_T$$
%		which proves the lower bound of the dynamic regret.
%		

\vskip 6pt

\nit{Scenario 2: $D \leq L_T \leq DT$.}
The proof is the same as the proof of Theorem \ref{thm: W>=1 lower bound} except for one difference: when $W=0$, we are able to give a better bound for $|J|$ even without the condition $L_T \geq 2D$.

% be identical to the proof of Theorem \ref{thm: W>=1 lower bound} in Section \ref{subsec: proof of lower bdd} except for one difference: when $W=0$, we are able to give a better bound for $|J|$ even without the condition $L_T \geq 2D$. Notice that the condition $D\leq L_T \leq DT$ is still necessary for the construction of  $\theta$ in Section \ref{subsec: proof of lower bdd} to be well-defined.

The bound for $|J|$ is given below. 
\begin{lemma} \label{lem: |J|}
	When  $T \geq 1$, and $D \leq L_T \leq DT$, we have 
	$ |\mathbb J| \geq \frac{L_T}{4D}$.
\end{lemma}
\begin{proof}
	
	By definition of $\mathbb J$ and 
	$\Delta  =\ceil{T/\floor{L_T/D}}\leq T/\floor{L_T/D}+1$,  when $L_T\geq D$ and $T\geq 1$, we have
	\begin{align*}
	|\mathbb J| &= \ceil{\frac{T}{\Delta}} \geq  \frac{T}{\Delta} \geq \frac{T}{T/\floor{L_T/D}+1}\\
	& = \floor{L_T/D}\frac{T}{T+\floor{L_T/D}} \geq \frac{L_T}{2D}\frac{T}{T+T} 
	 =\frac{L_T}{4D}
	\end{align*}
by $\floor{x} \geq x/2$ when $x\geq 1$, and $L_T \leq DT$.
\end{proof}

Then, the lower bound on the expected regret can be proved in the same way as in the proof of Theorem \ref{thm: W>=1 lower bound}.
\begin{align*}
&\E[\mathsf{C_T}(\bm x^{\A})-\mathsf{C_T}(\bm x^*)] \geq\E \frac{\alpha}{2} \|\bm x^{\A}-\bm x^*\|^2  \geq  \frac{\alpha}{2}\sum_{t\in \mathbb J} \frac{a_{t,t}^2 D^2}{4}  
 \geq \frac{\alpha D L_T}{32} (1-\rho)^2 \left(\frac{\alpha}{\alpha+ \beta}\right)^2
\end{align*}

\subsection{Additional numerical results for the last paragraph of Section III}\label{aped: rhgd recent}
In the last paragraph of Section III, we mention that our RHGD and RHAG do not use the latest information, i.e. RHGD computes $x_{s}(k)$ with $x_{s-1}(k-1)$ in Step 2 even though $x_{s-1}(k)$ is available (RHAG is similar). In this section, we provide numerical results that compare RHGD and RHAG with their variants that use the latest information, i.e. computing $x_{s}(k)$ with $x_{s-1}(k)$ in Step 2 of RHGD (RHAG is similar).  The setup of our numerical experiment is the same as Section VI.2.

%The setup of our numerical experiment is given below.
%We consider the cost function for logistic regression \cite{logisticregression} with a quadratic regularization term \cite{goel2019online}:
%$$ f_t(x_t)= \frac{1}{M}\sum_{m=1}^M\left[ \log\left( 1+ e^{ w_{t,m}^\top x_t }\right)-v_{t,m} w_{t,m}^\top x_t\right]+ \frac{r}{2}\|x_t\|^2$$
%where $M$ is the number of samples, $w_{t,m} \in R^N$ are features, $v_{t,m} \in \{0,1\}$ are labels, and $r$ is the regularization parameter. In particular, we consider $M=60$, $N=3$, $T=60$, $\beta=5$, and $r=0.5$. We generate $\{w_{t,m}\}_{m=1}^M$ as i.i.d. random Gaussian vectors with mean $\mu_t$ and variance $\sigma_t^2$, where  mean values $\mu_1, \dots, \mu_T$ are uniformly randomly drawn from $[-1,1]$, and variance $\sigma_1^2, \dots, \sigma_T^2$ are uniformly drawn from $[0,1]$. Similarly, we generate $\{v_{t,m}\}_{m=1}^M$ as i.i.d random values from Bernoulli distribution with parameter $p_t$, where $p_1, \dots, p_T$ are uniformly drawn from $[0,1]$. We consider initial estimator $x_0=0$ and feasible region $X=[-1,1]^3$.  % {stepsize are not exactly the factors.} 
%The stepsizes are based on the strong convexity factor and the smoothness factor. 
\begin{figure}
	\centering
	\includegraphics[width=0.5\linewidth]{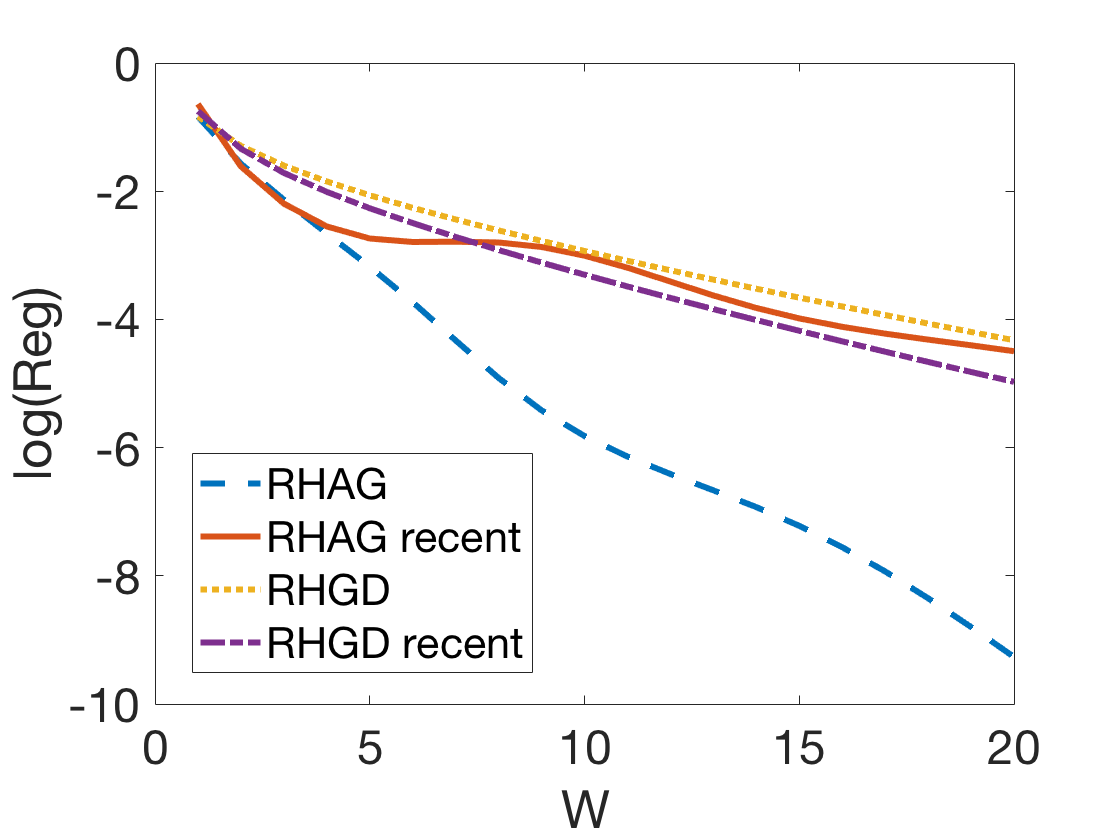}
	\caption{Comparison between RHGD, RHAG and their variants that use the most recent information.}
	\label{fig:rhagrecent}
\end{figure}

In Figure \ref{fig:rhagrecent}, we compare RHGD, RHAG with their variants that use the most recent information, i.e. RHGD-recent and RHAG-recent. It can be observed that by utilizing the most recent information, the performance of RHGD-recent and RHGD are similar. However, RHAG-recent performs much worse than RHAG. Some  intuitive explanation is provided below. Since RHAG implements  the offline Nesterov's accelerated gradient (NAG) exactly, RHAG-recent can be viewed as the offline NAG with some errors/disturbances in the inputs at each iteration. It is well-known that NAG is  sensitive to disturbances and small disturbances can worsen the performance a lot \cite{devolder2014first}. This might explain why  RHAG-recent performs much worse than RHAG.

%In this section, we provide numerical experiments to c

\bibliographystyle{IEEEtran}
\bibliography{citation4OL}

\begin{thebibliography}{10}
\providecommand{\url}[1]{#1}
\csname url@samestyle\endcsname
\providecommand{\newblock}{\relax}
\providecommand{\bibinfo}[2]{#2}
\providecommand{\BIBentrySTDinterwordspacing}{\spaceskip=0pt\relax}
\providecommand{\BIBentryALTinterwordstretchfactor}{4}
\providecommand{\BIBentryALTinterwordspacing}{\spaceskip=\fontdimen2\font plus
\BIBentryALTinterwordstretchfactor\fontdimen3\font minus
  \fontdimen4\font\relax}
\providecommand{\BIBforeignlanguage}[2]{{%
\expandafter\ifx\csname l@#1\endcsname\relax
\typeout{** WARNING: IEEEtran.bst: No hyphenation pattern has been}%
\typeout{** loaded for the language `#1'. Using the pattern for}%
\typeout{** the default language instead.}%
\else
\language=\csname l@#1\endcsname
\fi
#2}}
\providecommand{\BIBdecl}{\relax}
\BIBdecl

\bibitem{lin2012online}
M.~Lin, Z.~Liu, A.~Wierman, and L.~L. Andrew, ``Online algorithms for
  geographical load balancing,'' in \emph{2012 international green computing
  conference (IGCC)}.\hskip 1em plus 0.5em minus 0.4em\relax IEEE, 2012, pp.
  1--10.

\bibitem{tanaka2006real}
M.~Tanaka, ``Real-time pricing with ramping costs: A new approach to managing a
  steep change in electricity demand,'' \emph{Energy Policy}, vol.~34, no.~18,
  pp. 3634--3643, 2006.

\bibitem{zenke2017continual}
F.~Zenke, B.~Poole, and S.~Ganguli, ``Continual learning through synaptic
  intelligence,'' in \emph{Proceedings of the 34th International Conference on
  Machine Learning-Volume 70}, 2017, pp. 3987--3995.

\bibitem{rios2016survey}
J.~Rios-Torres and A.~A. Malikopoulos, ``A survey on the coordination of
  connected and automated vehicles at intersections and merging at highway
  on-ramps,'' \emph{IEEE Transactions on Intelligent Transportation Systems},
  vol.~18, no.~5, pp. 1066--1077, 2016.

\bibitem{hazan2016introduction}
E.~Hazan, \emph{Introduction to Online Convex Optimization}.\hskip 1em plus
  0.5em minus 0.4em\relax Now Publishers, 2016.

\bibitem{chen2016using}
N.~Chen, J.~Comden, Z.~Liu, A.~Gandhi, and A.~Wierman, ``Using predictions in
  online optimization: Looking forward with an eye on the past,'' in
  \emph{Proceedings of the 2016 ACM SIGMETRICS International Conference on
  Measurement and Modeling of Computer Science}.\hskip 1em plus 0.5em minus
  0.4em\relax ACM, 2016, pp. 193--206.

\bibitem{rawlings2012postface}
J.~Rawlings and D.~Mayne, ``Postface to model predictive control: Theory and
  design,'' \emph{Nob Hill Pub}, pp. 155--158, 2012.

\bibitem{alessio2009survey}
A.~Alessio and A.~Bemporad, ``A survey on explicit model predictive control,''
  in \emph{Nonlinear model predictive control}.\hskip 1em plus 0.5em minus
  0.4em\relax Springer, 2009, pp. 345--369.

\bibitem{kogel2014stabilization}
M.~K{\"o}gel and R.~Findeisen, ``Stabilization of inexact {MPC} schemes,'' in
  \emph{53rd IEEE Conference on Decision and Control}.\hskip 1em plus 0.5em
  minus 0.4em\relax IEEE, 2014, pp. 5922--5928.

\bibitem{wang2010fast}
Y.~Wang and S.~Boyd, ``Fast model predictive control using online
  optimization,'' \emph{IEEE Transactions on Control Systems Technology},
  vol.~18, no.~2, pp. 267--278, 2010.

\bibitem{graichen2010stability}
K.~Graichen and A.~Kugi, ``Stability and incremental improvement of suboptimal
  {MPC} without terminal constraints,'' \emph{IEEE Transactions on Automatic
  Control}, vol.~55, no.~11, pp. 2576--2580, 2010.

\bibitem{paternain2018prediction}
S.~Paternain, M.~Morari, and A.~Ribeiro, ``A prediction-correction method for
  model predictive control,'' in \emph{2018 Annual American Control Conference
  (ACC)}.\hskip 1em plus 0.5em minus 0.4em\relax IEEE, 2018, pp. 4189--4194.

\bibitem{diehl2005nominal}
M.~Diehl, R.~Findeisen, F.~Allg{\"o}wer, H.~G. Bock, and J.~P. Schl{\"o}der,
  ``Nominal stability of real-time iteration scheme for nonlinear model
  predictive control,'' \emph{IEEE Proceedings-Control Theory and
  Applications}, vol. 152, no.~3, pp. 296--308, 2005.

\bibitem{mokhtari2016online}
A.~Mokhtari, S.~Shahrampour, A.~Jadbabaie, and A.~Ribeiro, ``Online
  optimization in dynamic environments: Improved regret rates for strongly
  convex problems,'' in \emph{55th IEEE Conference on Decision and Control
  (CDC)}.\hskip 1em plus 0.5em minus 0.4em\relax IEEE, 2016, pp. 7195--7201.

\bibitem{ellis2014economic}
M.~Ellis and P.~D. Christofides, ``Economic model predictive control with
  time-varying objective function for nonlinear process systems,'' \emph{AIChE
  Journal}, vol.~60, no.~2, pp. 507--519, 2014.

\bibitem{ferramosca2014economic}
A.~Ferramosca, D.~Limon, and E.~F. Camacho, ``Economic {MPC} for a changing
  economic criterion for linear systems,'' \emph{IEEE Transactions on Automatic
  Control}, vol.~59, no.~10, pp. 2657--2667, 2014.

\bibitem{angeli2016theoretical}
D.~Angeli, A.~Casavola, and F.~Tedesco, ``Theoretical advances on economic
  model predictive control with time-varying costs,'' \emph{Annual Reviews in
  Control}, vol.~41, pp. 218--224, 2016.

\bibitem{alessandretti2016convergence}
A.~Alessandretti, A.~P. Aguiar, and C.~N. Jones, ``On convergence and
  performance certification of a continuous-time economic model predictive
  control scheme with time-varying performance index,'' \emph{Automatica},
  vol.~68, pp. 305--313, 2016.

\bibitem{grune2018economic}
L.~Gr{\"u}ne and S.~Pirkelmann, ``Economic model predictive control for
  time-varying system: Performance and stability results,'' \emph{Optimal
  Control Applications and Methods}, 2018.

\bibitem{zavala2010real}
V.~M. Zavala and M.~Anitescu, ``Real-time nonlinear optimization as a
  generalized equation,'' \emph{SIAM Journal on Control and Optimization},
  vol.~48, no.~8, pp. 5444--5467, 2010.

\bibitem{zanon2013lyapunov}
M.~Zanon, S.~Gros, and M.~Diehl, ``A {L}yapunov function for periodic economic
  optimizing model predictive control,'' in \emph{52nd IEEE Conference on
  Decision and control}.\hskip 1em plus 0.5em minus 0.4em\relax IEEE, 2013, pp.
  5107--5112.

\bibitem{li2019online}
Y.~Li, X.~Chen, and N.~Li, ``Online optimal control with linear dynamics and
  predictions: Algorithms and regret analysis,'' in \emph{Advances in Neural
  Information Processing Systems}, 2019, pp. 14\,858--14\,870.

\bibitem{goel2019online}
G.~Goel and A.~Wierman, ``An online algorithm for smoothed regression and {LQR}
  control,'' \emph{Proceedings of Machine Learning Research}, vol.~89, pp.
  2504--2513, 2019.

\bibitem{rakhlin2013online}
A.~Rakhlin and K.~Sridharan, ``Online learning with predictable sequences,'' in
  \emph{Conference on Learning Theory}, 2013, pp. 993--1019.

\bibitem{besbes2015non}
O.~Besbes, Y.~Gur, and A.~Zeevi, ``Non-stationary stochastic optimization,''
  \emph{Operations research}, vol.~63, no.~5, pp. 1227--1244, 2015.

\bibitem{andrew2013tale}
L.~Andrew, S.~Barman, K.~Ligett, M.~Lin, A.~Meyerson, A.~Roytman, and
  A.~Wierman, ``A tale of two metrics: Simultaneous bounds on competitiveness
  and regret,'' in \emph{Conference on Learning Theory}, 2013, pp. 741--763.

\bibitem{simonetto2016class}
A.~Simonetto, A.~Mokhtari, A.~Koppel, G.~Leus, and A.~Ribeiro, ``A class of
  prediction-correction methods for time-varying convex optimization,''
  \emph{IEEE Transactions on Signal Processing}, vol.~64, no.~17, pp.
  4576--4591, 2016.

\bibitem{tang2018running}
Y.~Tang, E.~Dall'Anese, A.~Bernstein, and S.~Low, ``Running primal-dual
  gradient method for time-varying nonconvex problems,'' \emph{arXiv preprint
  arXiv:1812.00613}, 2018.

\bibitem{nesterov2013introductory}
Y.~Nesterov, \emph{Introductory lectures on convex optimization: A basic
  course}.\hskip 1em plus 0.5em minus 0.4em\relax Springer Science \& Business
  Media, 2013, vol.~87.

\bibitem{rosales1998improved}
R.~Rosales and S.~Sclaroff, ``Improved tracking of multiple humans with
  trajectory prediction and occlusion modeling,'' Boston University Computer
  Science Department, Tech. Rep., 1998.

\bibitem{pentina2015curriculum}
A.~Pentina, V.~Sharmanska, and C.~H. Lampert, ``Curriculum learning of multiple
  tasks,'' in \emph{Proceedings of the IEEE Conference on Computer Vision and
  Pattern Recognition}, 2015, pp. 5492--5500.

\bibitem{li2018online}
Y.~Li, G.~Qu, and N.~Li, ``Online optimization with predictions and switching
  costs: Fast algorithms and the fundamental limit,'' \emph{arXiv preprint
  arXiv:1801.07780}, 2018.

\bibitem{devolder2014first}
O.~Devolder, F.~Glineur, and Y.~Nesterov, ``First-order methods of smooth
  convex optimization with inexact oracle,'' \emph{Mathematical Programming},
  vol. 146, no. 1-2, pp. 37--75, 2014.

\bibitem{logisticregression}
\BIBentryALTinterwordspacing
(2012) Logistic regression. [Online]. Available:
  \url{http://www.stat.cmu.edu/\~cshalizi/uADA/12/lectures/ch12.pdf}
\BIBentrySTDinterwordspacing

\bibitem{su2014differential}
W.~Su, S.~Boyd, and E.~Candes, ``A differential equation for modeling
  nesterov's accelerated gradient method: Theory and insights,'' in
  \emph{Advances in Neural Information Processing Systems}, 2014, pp.
  2510--2518.

\bibitem{durrett2019probability}
R.~Durrett, \emph{Probability: theory and examples}.\hskip 1em plus 0.5em minus
  0.4em\relax Cambridge university press, 2019, vol.~49.

\bibitem{varah1975lower}
J.~M. Varah, ``A lower bound for the smallest singular value of a matrix,''
  \emph{Linear Algebra and its Applications}, vol.~11, no.~1, pp. 3--5, 1975.

\bibitem{concus1985block}
P.~Concus, G.~H. Golub, and G.~Meurant, ``Block preconditioning for the
  conjugate gradient method,'' \emph{SIAM Journal on Scientific and Statistical
  Computing}, vol.~6, no.~1, pp. 220--252, 1985.

\end{thebibliography}

%\IEEEpeerreviewmaketitle

%\input{sec1_introduction}
%%%\input{circles}
%\input{sec2_problem_formulation}
%\input{sec3_prev_alg}
%\input{sec4_my_alg}
%\input{sec5_performance}
%\input{sec6_lowerbdd}
%\input{sec7_simulation}
%\input{sec8_conclusion}
%\input{appendix}
%\input{appendixB}
%

%\section*{Acknowledgment}


\ifCLASSOPTIONcaptionsoff
  \newpage
\fi

\end{document}